\documentclass[11pt]{amsart}
\usepackage {amsmath, amssymb, amscd, mathrsfs, amsthm, stmaryrd,  bbm, diagbox, enumerate, slashed, graphicx, color, subfig, transparent, comment, float}
\usepackage[dvipsnames]{xcolor}
\usepackage{url, hyperref}
\usepackage{footmisc}
\usepackage{tikz-cd}
\usepackage[margin=1.2in]{geometry}

\newtheorem{theorem}{Theorem}[section]
\newtheorem{proposition}[theorem]{Proposition}
\newtheorem{corollary}[theorem]{Corollary}
\newtheorem{lemma}[theorem]{Lemma}

\newtheorem{definition}[theorem]{Definition}
\theoremstyle{remark}
\newtheorem{example}[theorem]{Example}
\theoremstyle{definition}

\newtheorem {remark}[theorem]{Remark}

\def\interior{\operatorname{int}}

\begin{document}
\title{An RBG construction of integral surgery homeomorphisms}
\author{Qianhe Qin}
\address {Department of Mathematics, Stanford University\\
Stanford, California, 94305, United States of America}
\email {\href{mailto:qqhe@stanford.edu}{qqhe@stanford.edu}}

\begin{abstract}
We generalize the RBG construction of Manolescu and Piccirillo to produce pairs of knots with the same $n$-surgery, and investigate the possibility of constructing exotic definite four-manifolds using $n$-surgery homeomorphisms.
\end{abstract}

\maketitle

\section{Introduction}

In \cite{MP21}, Manolescu and Piccirillo introduced RBG links, a kind of $3$-component framed links in $S^3$ that produce knot pairs with the same $0$-surgery. (Similar constructions appeared earlier in Akbulut's work \cite{akbulut_1993},\cite{akbulut_1977}.) RBG links are relevant for an approach to constructing exotic definite 4-manifolds. The strategy is to find a knot pair $(K,K^{\prime})$, such that $S_0^3(K)\cong S_0^3(K^{\prime})$, $K$ is $H$-slice in some 4-manifold $W$ (i.e. bounds a null-homologous disk in $W\setminus \interior B^4$) and $K^{\prime}$ is not $H$-slice in $W$. Then, one can construct a new 4-manifold $W^{\prime}$ (an exotic copy of $W$) by carving out a neighborhood of the slice disk bounded by $K$, and gluing back the trace of $0$-surgery on $K^{\prime}$ using some $0$-surgery homeomorphism.

In \cite{MP21}, they focused on a class of RBG links called special RBG links, and experimented on a $6$-parameter family of RBG links. They used Rasmussen's $s$-invariant to obstruct $K^{\prime}$ from being $H$-slice in $W$, and collected several knots $K$ where the usual invariants obstructing $H$-sliceness vanish. Later on, however, Nakamura showed that these knots $K$ are not slice. He developed a method in \cite{Kai} to stably relate the traces of $K$ and $K'$, and obstruct $K$ from being $H$-slice using $s(K^{\prime})\neq 0$. 

In this paper, for $n\in\mathbb{Z}$, we generalize RBG links to $|n|$-RBG links, which can be used to produce knot pairs $(K,J)$ such that $S_l^3(K)\cong S^3_{m}(J)$ with $l,m\in \{n,-n\}$.

\begin{definition}
\label{def:RBG}
An \emph{$|n|$-RBG link} $L=\{(R,r),(B,b),(G,g)\}$ is a $3$-component framed link in $S^3$, with framings $r\in\mathbb{Q}$ and $b,g\in\mathbb{Z}$, together with homeomorphisms $\psi_B: S^3_{r,g}(R,G)\rightarrow S^3$ and $\psi_G: S^3_{r,b}(R,B)\rightarrow S^3$, such that $H_1(S^3_{r,b,g}(R,B,G); \mathbb{Z})=\mathbb{Z} /n\mathbb{Z}$. 
\end{definition}
\begin{remark}
A certain type of $|n|$-RBG links was defined in the context of Legendrian knots in \cite[Definition 3.19]{CEK}.
\end{remark}
\begin{theorem}
\label{thm:RBG}
Any $|n|$-RBG link $L=\{(R,r),(B,b),(G,g)\}$ has an associated knot pair $(K_B, K_G)$ and a homeomorphism $\phi_L: S^3_{f_b}(K_B) \rightarrow S^3_{f_g}(K_G)$ with $f_b,f_g\in\{n,-n\}$. Conversely, given a homeomorphism $\phi: S_l^3(K)\rightarrow S^3_{m}(J)$ with $l,m\in\{n,-n\}$, there exists an $|n|$-RBG link $L_{\phi}$ such that the associated knot pair is $(K,J)$ and $\phi_{L_{\phi}}=\phi$ up to isotopy.
\end{theorem}

As in \cite{MP21}, one can attempt to use $n$-surgery homeomorphisms to construct exotic $4$-manifolds. A knot $K$ is said to be \emph{$n$-slice} in $W$, if $K\subset \partial (W\setminus \interior B^4)$ bounds a properly embedded disk $D$ with self-intersection number $-n$. If there is another knot $K^{\prime}$ with $S^3_n(K)\cong S^3_n(K^{\prime})$, by removing a tubular neighborhood of $D$ and gluing back the trace of $n$-surgery on $K^{\prime}$, we obtain a new $4$-manifold $W^{\prime}$ such that $K^{\prime}$ is $n$-slice in $W^{\prime}$.

For $W=\#^l \overline{\mathbb{CP}^2}$, we have that $W^{\prime}$ is homeomorphic to $W$. If $K^{\prime}$ is not $n$-slice in $W$, then $W^{\prime}$ is not diffeomorphic to $W$. Moreover, there is an adjunction inequality for the Rasmussen's $s$-invariant for knots which are $n$-slice in $W$; this was conjectured in \cite{ff} and was proved by Ren in \cite{Fpq}. Thus, one can try to use the $s$-invariant and a pair of knots with the same $n$-surgery to construct an exotic $\#^l \overline{\mathbb{CP}^2}$.

\subsection{n-special RBG links}
We define $n$-special RBG links, for which the associated knot pairs are easier to find diagrammatically and the associated knots have the same $n$-surgery. 

\begin{definition}
\label{def:special}
A link $L=\{(R,r),(B,b),(G,g)\}$ with linking matrix $M_L$, is called an \emph{$n$-special RBG link}, if 
 $b=g=0$, $n=-det(M_L)$, and there exist link isotopies 
 $$R\cup B\cong R\cup \mu_R \cong R\cup G,$$where $\mu_R$ is the meridian of $R$.
\end{definition}

As in \cite{MP21}, one can experiment on parametrized families of $n$-special RBG links to look for knot pairs that share the same $n$-surgery, such that one of the knots is $n$-slice in some $\#^l \overline{\mathbb{CP}^2}$ and the other is not $n$-slice in any  $\#^m \overline{\mathbb{CP}^2}$. 

In \cite{dd}, Nakamura showed that for special RBG links such that $R=U$ and $K_B$ is $H$-slice in $W$, Rasmussen's $s$-invariant cannot be used to obstruct $K_G$ from being $H$-slice. We generalize Nakamura's theorem to $n$-special RBG links.

\begin{theorem}
\label{thm:sinv}
Let $L=\{(R,r),(B,0),(G,0)\}$ be an $n$-special RBG link with $n$ nonnegative. 
\begin{enumerate}[(a)]
\item  If $R$ is $r$-slice in some $\#^m {\mathbb{CP}^2}$ and $K_B$ is $n$-slice in $\#^l \overline{\mathbb{CP}^2}$, then $s(K_G)\le n-\sqrt{n}$.

\item If $R$ is $(r-1)$-slice in some $\#^m \overline{\mathbb{CP}^2}$, then $s(K_G)\le n+1-\sqrt{n+1}$.
\end{enumerate}
\end{theorem}

If $n > 0$, the above theorem leaves open the possibility of using the $s$-invariant to detect exotic pairs of definite $4$-manifolds from $n$-special RBG links where $R$ is the unknot. For example, we have the following:
\begin{theorem}
\label{potential3}
If the knot $K(-2,1,3)$ from the left-hand side of Figure \ref{intro} is $3$-slice in some $\#^m \overline{\mathbb{CP}^2}$, then there exists an exotic $\#^m \overline{\mathbb{CP}^2}$.
\end{theorem}

Theorem \ref{thm:sinv} also gives a new way of obstructing knots from being $n$-slice, by finding another knot with the same $n$-surgery and checking its $s$-invariant. (This generalizes the $n=0$ method, which was used by Piccirillo in her proof that the Conway knot is not slice in \cite{Conwaylisa}, and then extended by Nakamura \cite{Kai}.) For example, we prove the following theorem.
\begin{figure}[h]
{
   \fontsize{9pt}{11pt}\selectfont
   \def\svgwidth{4in}
   \begin{center}
   %% Creator: Inkscape 1.3.2 (091e20e, 2023-11-25), www.inkscape.org
%% PDF/EPS/PS + LaTeX output extension by Johan Engelen, 2010
%% Accompanies image file 'introduction.pdf' (pdf, eps, ps)
%%
%% To include the image in your LaTeX document, write
%%   \input{<filename>.pdf_tex}
%%  instead of
%%   \includegraphics{<filename>.pdf}
%% To scale the image, write
%%   \def\svgwidth{<desired width>}
%%   \input{<filename>.pdf_tex}
%%  instead of
%%   \includegraphics[width=<desired width>]{<filename>.pdf}
%%
%% Images with a different path to the parent latex file can
%% be accessed with the `import' package (which may need to be
%% installed) using
%%   \usepackage{import}
%% in the preamble, and then including the image with
%%   \import{<path to file>}{<filename>.pdf_tex}
%% Alternatively, one can specify
%%   \graphicspath{{<path to file>/}}
%% 
%% For more information, please see info/svg-inkscape on CTAN:
%%   http://tug.ctan.org/tex-archive/info/svg-inkscape
%%
\begingroup%
  \makeatletter%
  \providecommand\color[2][]{%
    \errmessage{(Inkscape) Color is used for the text in Inkscape, but the package 'color.sty' is not loaded}%
    \renewcommand\color[2][]{}%
  }%
  \providecommand\transparent[1]{%
    \errmessage{(Inkscape) Transparency is used (non-zero) for the text in Inkscape, but the package 'transparent.sty' is not loaded}%
    \renewcommand\transparent[1]{}%
  }%
  \providecommand\rotatebox[2]{#2}%
  \newcommand*\fsize{\dimexpr\f@size pt\relax}%
  \newcommand*\lineheight[1]{\fontsize{\fsize}{#1\fsize}\selectfont}%
  \ifx\svgwidth\undefined%
    \setlength{\unitlength}{513.5724742bp}%
    \ifx\svgscale\undefined%
      \relax%
    \else%
      \setlength{\unitlength}{\unitlength * \real{\svgscale}}%
    \fi%
  \else%
    \setlength{\unitlength}{\svgwidth}%
  \fi%
  \global\let\svgwidth\undefined%
  \global\let\svgscale\undefined%
  \makeatother%
  \begin{picture}(1,0.50086667)%
    \lineheight{1}%
    \setlength\tabcolsep{0pt}%
    \put(0,0){\includegraphics[width=\unitlength,page=1]{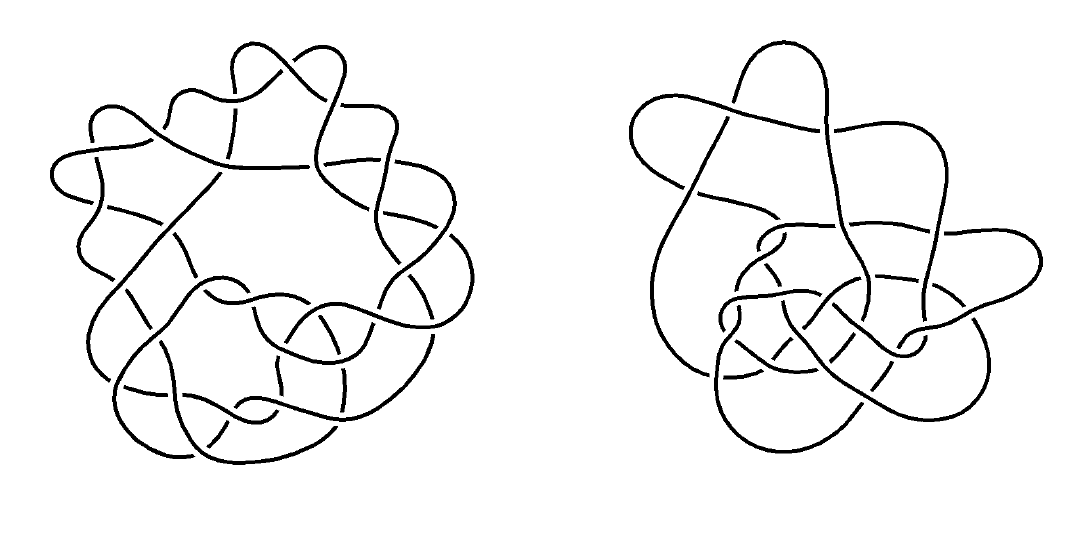}}%
    \put(0.17575014,0.00462935){\color[rgb]{0,0,0}\makebox(0,0)[lt]{\lineheight{1.25}\smash{\begin{tabular}[t]{l}$K(-2,1,3)$\end{tabular}}}}%
    \put(0.69560772,0.00462931){\color[rgb]{0,0,0}\makebox(0,0)[lt]{\lineheight{1.25}\smash{\begin{tabular}[t]{l}$K(-2,1,2)$\end{tabular}}}}%
  \end{picture}%
\endgroup%

   \end{center}
   \caption{}
   \label{intro}
}
\end{figure}
\begin{theorem} 
\label{thm:2slice}
The knots in Figure \ref{intro} are not $2$-slice in any $\#^m \overline{\mathbb{CP}^2}$.
\end{theorem}

\subsection{n-peculiar RBG links}

We will also define a different class of $|n|$-RBG links called \emph{$n$-peculiar RBG links} (see Section \ref{npeculiar}), for which the red components $R$ are rationally framed and the associated knot pairs can be obtained diagrammatically. 

\begin{definition}
\label{def:pecu}
A link $\{(R,r),(B,b),(G,g)\}$ is called an $n$-peculiar RBG link, if there exists $t\in\mathbb{Z}$ such that
\begin{itemize}
\item $R=U$ and $B,G$ are meridians of $R$,
\item $b=g= 1/r + 1/t$, 
\item $n=(g+b-2l) - t(l-b)^2$, 
\end{itemize}
where $l=lk(B,G)$ under an orientation of $L$ such that $lk(B,R)=lk(G,R)=1$.
\end{definition}
This gives a new construction of RBG links when $n=0$, for which Nakamura's obstruction in \cite{Kai} does not immediately apply; so, in principle, they can potentially produce exotic $4$-spheres.

\medskip
{\bf Organization of the paper.} In Section \ref{nsurg}, we generalize the RBG construction of zero surgeries to integral surgeries. We discuss the construction of a potential exotic pair by cutting and pasting of $n$-traces. In Section \ref{nspecial}, we introduce $n$-special RBG links, for which the associated knot pairs can be obtained diagrammatically. In Section \ref{nsinv}, we generalize Nakamura's sliceness obstruction to an $n$-sliceness obstruction using $n$-special RBG links, and we give examples where the $n$-sliceness of knots in $\#^m \overline{\mathbb{CP}^2}$ is obstructed by this new method. In Section \ref{npeculiar}, we discuss $n$-peculiar RBG links.

\medskip
{\bf Acknowledgements.} The author is grateful to her advisor, Ciprian Manolescu, for his suggestion to work on this project and his continued guidance and support. We would also like to thank Anthony Conway, Nathan Dunfield, Kyle Hayden, Maggie Miller, Lisa Piccirillo, Qiuyu Ren, Ali Naseri Sadr and Charles Stine for helpful suggestions and conversations.

\section{$\langle n\rangle$-surgery homeomorphisms}
\label{nsurg}
Let $n$ be an integer and $|n|$ be its absolute value. We use $\langle n\rangle$ to denote the set $\{n,-n\}$.
{
\renewcommand{\thetheorem}{\ref{def:RBG}}
\begin{definition}
An \emph{$|n|$-RBG link} $L=\{(R,r),(B,b),(G,g)\}$ is a $3$-component framed link in $S^3$, with framings $r\in\mathbb{Q}$ and $b,g\in\mathbb{Z}$, together with homeomorphisms $\psi_B: S^3_{r,g}(R,G)\rightarrow S^3$ and $\psi_G: S^3_{r,b}(R,B)\rightarrow S^3$, such that $H_1(S^3_{r,b,g}(R,B,G); \mathbb{Z})=\mathbb{Z} /n\mathbb{Z}$. 
\end{definition}
\addtocounter{theorem}{-1}
}

\begin{remark}
In \cite{MP21}, RBG links are defined to be rationally framed for 0-surgeries. In the case of $\pm n$-surgeries with $n\neq 0$, we restrict the framings $b,g$ of $B,G$ to be integers, so that we can pin down the surgery coefficient to $\pm n$ from the homological condition: $H_1(S^3_{r,b,g}(R,B,G); \mathbb{Z})=\mathbb{Z} /n\mathbb{Z}$.  
\end{remark}

\begin{definition}
Given a pair of framed knots $\{(K,f_K),(J,f_J)\}$, a homeomorphism $\phi: S_{f_K}^3(K)\rightarrow S^3_{f_J}(J)$ is called an \emph{$\langle n\rangle$-surgery homeomorphism}, if $f_K,f_J\in \{n,-n\}$. If $f_K=f_J=n$, then we call $\phi$ an \emph{$n$-surgery homeomorphism}.
\end{definition}

We generalize Theorem 1.2 of \cite{MP21} to $|n|$-RBG links as in the following theorem, which is a rephrasing of Theorem \ref{thm:RBG} from the introduction.

\begin{theorem}
Any $|n|$-RBG link $L=\{(R,r),(B,b),(G,g)\}$ has an associated knot pair $(K_B, K_G)$ and an $\langle n\rangle$-surgery homeomorphism $\phi_L: S^3_{f_b}(K_B) \rightarrow S^3_{f_g}(K_G)$. Conversely, given an $\langle n\rangle$-surgery homeomorphism $\phi: S_l^3(K)\rightarrow S^3_{m}(J)$, there exists an $|n|$-RBG link $L_{\phi}$ such that the associated knot pair is $(K,J)$ and $\phi_{L_{\phi}}=\phi$ up to isotopy.
\end{theorem}

\begin{proof}

Given an $|n|$-RBG link $L=\{(R,r),(B,b),(G,g)\}$, we associate to it two framed knots $(K_B,f_b), (K_G,f_g)$ and an $\langle n\rangle$-surgery homeomorphism $\phi_L$ as follows: First, let $(K_B,f_b)$ be $\psi_B(B,b)$, and let $(K_G,f_g)$ be $\psi_G(G,g)$. Since a homeomorphism maps an integer framing to an integer framing, $f_b,f_g$ are integers. Then, extend $\psi_B$ (resp. $\psi_G$) to $\widetilde{\psi}_B: S^3_{r,b,g}(R,B,G) \rightarrow S^3_{f_b}(K_B)$ (resp. $\widetilde{\psi}_G: S^3_{r,b,g}(R,B,G) \rightarrow S^3_{f_g}(K_G)$) by gluing back tubular neighborhoods of $B$ and $K_B$ (resp. $G$ and $K_G$) according to the framings. Since $H_1(S^3_{r,b,g}(R,B,G); \mathbb{Z})=\mathbb{Z} /n\mathbb{Z}$, we have $f_b,f_g\in\{n,-n\}$. Finally, define $\phi_L$ as 
$$\phi_L:=\widetilde{\psi}_G\circ\widetilde{\psi}_B^{-1}:S^3_{f_b}(K_B)\rightarrow S^3_{f_g}(K_G).$$

Conversely, given an $\langle n\rangle$-surgery homeomorphism $\phi: S_l^3(K)\rightarrow S^3_{m}(J)$, define an $|n|$-RBG link $L_{\phi}$ as follows. 

Fix set-wise representatives of $S_l^3(K),S^3_{m}(J)$ by specifying the knots $K,J\subset S^3$ and the surgery tubular neighbourhoods $\nu(K),\nu(J)$. Pick a meridian $\mu_J\subset S^3\backslash \nu(J)$. Up to isotopy, we can assume that $\phi^{-1}$ maps $\mu_J$ into $S^3\backslash \nu(K)$. Choose a tubular neighborhood $Q$ of $J$ which contains $\mu_J$, and pick a tubular neighborhood $N$ of $\mu_J$ such that $N\subset Q\backslash \nu(J)$ and $\phi^{-1}(N)\subset S^3\backslash \nu(K)$. Pick a meridian $\mu_{\mu_J}$ of $\mu_J$ in $N$. (See Figure \ref{NQ}.) Let $L_{\phi}$ be $\{\phi^{-1}(\mu_J,0),(K,l),(\phi^{-1}(\mu_{\mu_J}),0)\}$.

\begin{figure}[h]
{
   \fontsize{9pt}{11pt}\selectfont
   \def\svgwidth{2.5in}
   \begin{center}
   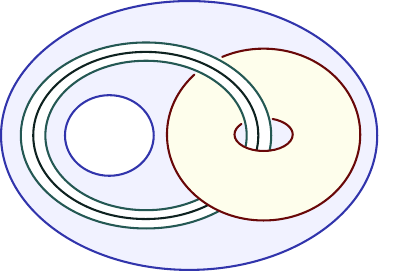
   \end{center}
   \caption{}
   \label{NQ}
}
\end{figure}
Let $\widetilde{N}$ be the manifold obtained by surgery on $N$ along $\{(\mu_{J},0),(\mu_{\mu_J},0)\}$, and let $\widetilde{Q}$ be the manifold obtained by surgery on $Q$ along $\{(\mu_{J},0),(J,m),(\mu_{\mu_J},0)\}$.
Extend $\phi:S_l^3(K)\rightarrow S^3_{m}(J)$ to $\widetilde{\phi}: S^3_{l,r,0}(K,R,G)\rightarrow S^3_{m,0,0}(J,\mu_{J},\mu_{\mu_J})$. Let $\widetilde{\psi}_B$ be the slam-dunk map (as in Figure 5.30 in \cite{kirbybook}) in $\widetilde{\phi}^{-1}(\widetilde{N})$ and the identity map on $S^3_l(K)\backslash \phi^{-1}(N)$. 
\begin{figure}[h]
{
   \fontsize{9pt}{11pt}\selectfont
   \def\svgwidth{2.5in}
   \begin{center}
   %% Creator: Inkscape 1.2.2 (b0a84865, 2022-12-01), www.inkscape.org
%% PDF/EPS/PS + LaTeX output extension by Johan Engelen, 2010
%% Accompanies image file 'nqslide.pdf' (pdf, eps, ps)
%%
%% To include the image in your LaTeX document, write
%%   \input{<filename>.pdf_tex}
%%  instead of
%%   \includegraphics{<filename>.pdf}
%% To scale the image, write
%%   \def\svgwidth{<desired width>}
%%   \input{<filename>.pdf_tex}
%%  instead of
%%   \includegraphics[width=<desired width>]{<filename>.pdf}
%%
%% Images with a different path to the parent latex file can
%% be accessed with the `import' package (which may need to be
%% installed) using
%%   \usepackage{import}
%% in the preamble, and then including the image with
%%   \import{<path to file>}{<filename>.pdf_tex}
%% Alternatively, one can specify
%%   \graphicspath{{<path to file>/}}
%% 
%% For more information, please see info/svg-inkscape on CTAN:
%%   http://tug.ctan.org/tex-archive/info/svg-inkscape
%%
\begingroup%
  \makeatletter%
  \providecommand\color[2][]{%
    \errmessage{(Inkscape) Color is used for the text in Inkscape, but the package 'color.sty' is not loaded}%
    \renewcommand\color[2][]{}%
  }%
  \providecommand\transparent[1]{%
    \errmessage{(Inkscape) Transparency is used (non-zero) for the text in Inkscape, but the package 'transparent.sty' is not loaded}%
    \renewcommand\transparent[1]{}%
  }%
  \providecommand\rotatebox[2]{#2}%
  \newcommand*\fsize{\dimexpr\f@size pt\relax}%
  \newcommand*\lineheight[1]{\fontsize{\fsize}{#1\fsize}\selectfont}%
  \ifx\svgwidth\undefined%
    \setlength{\unitlength}{203.56605271bp}%
    \ifx\svgscale\undefined%
      \relax%
    \else%
      \setlength{\unitlength}{\unitlength * \real{\svgscale}}%
    \fi%
  \else%
    \setlength{\unitlength}{\svgwidth}%
  \fi%
  \global\let\svgwidth\undefined%
  \global\let\svgscale\undefined%
  \makeatother%
  \begin{picture}(1,0.6388769)%
    \lineheight{1}%
    \setlength\tabcolsep{0pt}%
    \put(0,0){\includegraphics[width=\unitlength,page=1]{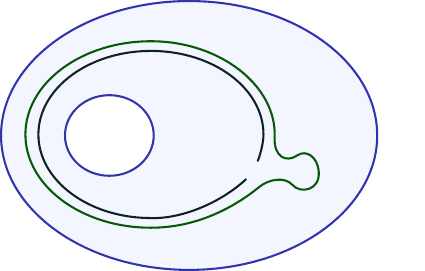}}%
    \put(0.75586615,0.02333059){\color[rgb]{0,0.01568627,0.59215686}\transparent{0.80000001}\makebox(0,0)[lt]{\lineheight{1.25}\smash{\begin{tabular}[t]{l}$Q$\end{tabular}}}}%
    \put(0.65923162,0.45048033){\color[rgb]{0.93333333,0.16470588,0.16470588}\makebox(0,0)[lt]{\lineheight{1.25}\smash{\begin{tabular}[t]{l}$(\mu_J,0)$\end{tabular}}}}%
    \put(0.62634676,0.1389437){\color[rgb]{0.03137255,0.34509804,0.02352941}\makebox(0,0)[lt]{\lineheight{1.25}\smash{\begin{tabular}[t]{l}$(\mu_{\mu_J},0)$\end{tabular}}}}%
    \put(0.32903218,0.5642391){\color[rgb]{0.05490196,0.13333333,0.1254902}\makebox(0,0)[lt]{\lineheight{1.25}\smash{\begin{tabular}[t]{l}$(J,m)$\end{tabular}}}}%
    \put(0,0){\includegraphics[width=\unitlength,page=2]{nqslide.pdf}}%
    \put(0.32811554,0.09976699){\color[rgb]{0.05490196,0.13333333,0.1254902}\makebox(0,0)[lt]{\lineheight{1.25}\smash{\begin{tabular}[t]{l}$m$\end{tabular}}}}%
  \end{picture}%
\endgroup%

   \end{center}
   \caption{}
   \label{nqslide}
}
\end{figure}

Let $\eta$ be the composition map ${\phi}\circ \widetilde{\psi}_B\circ (\widetilde{\phi})^{-1}$, which is identity outside of $\widetilde{N}$. Slide $\mu_{\mu_J}$ over $J$ and cancel the pair $(J,\mu_J)$ in $\widetilde{Q}$ (as in Figure \ref{nqslide}). Together, they induce a homeomorphism $\psi$ which is identity on $S^3\backslash Q$. Let $\widetilde{\psi}_G$ be $\psi\circ \widetilde{\phi}$. Thus, we have the following commutative diagram:
\begin{center}
\begin{tikzcd}
 & S^3_{r,l,0}(R,B,G) \arrow[r, "\widetilde{\psi}_B"] \arrow[d, "\widetilde{\phi}"] \arrow[dl,"\widetilde{\psi}_G", bend right=20]
    & S_l^3(K) \arrow[d, "\phi"] \\
S^3_{m}(J) & S^3_{0,m,0}(\mu_J,J,\mu_{\mu_J}) \arrow[r, "\eta"] \arrow[l, "\psi"]
& S^3_{m}(J) 
\end{tikzcd}
\end{center}

Since $\psi\circ\eta^{-1}$ is identity outside of ${Q}$ and $MCG(S^1\times D^2, S^1\times S^1)$ is trivial, $\psi\circ\eta^{-1}$ is isotopic to the identity. Therefore, $\widetilde{\psi}_G \circ \widetilde{\psi}_B^{-1}$ is isotopic to $\phi$. Undo the surgery on $B$ (resp. $G$), we obtain $\psi_B$ (resp. $\psi_G$) from $\widetilde{\psi}_B$ (resp. $\widetilde{\psi}_G$).
\end{proof}

\begin{example}
Consider a homeomorphism $\phi$ between the $1$-surgery on the figure-eight knot $K$ and the $(-1)$-surgery on the right-handed trefoil $J$ in Figure \ref{KA}, which is an analogue of Figure 23 in \cite{KA}.

\begin{figure}[h]
{
   \fontsize{9pt}{11pt}\selectfont
   \def\svgwidth{3.5in}
   \begin{center}
   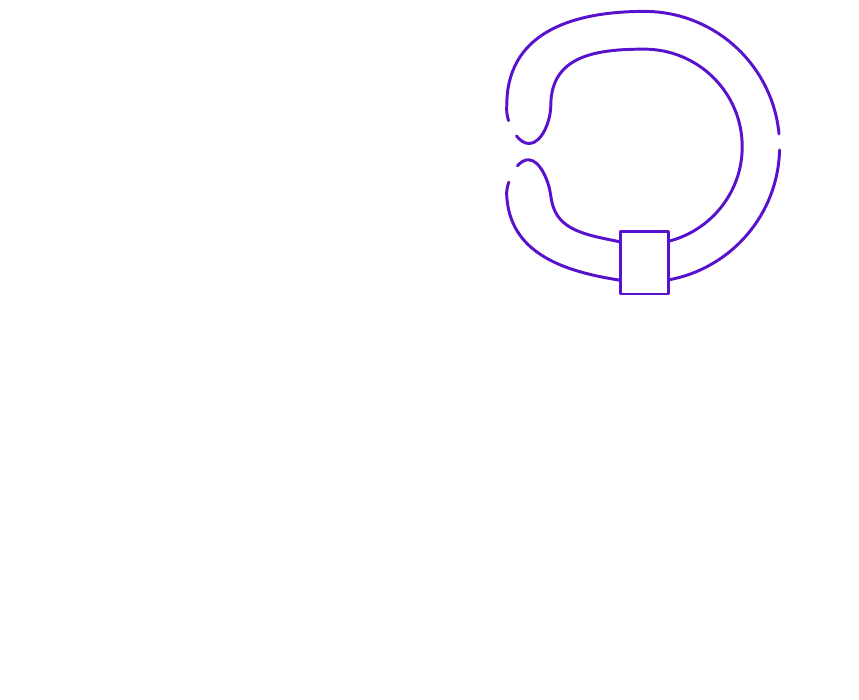
   \end{center}
   \caption{}
   \label{KA}
}
\end{figure}
We can construct the corresponding $|1|$-RBG link (Figure \ref{pmlink}) by chasing the image of $(\mu_J,0)$ under the map $\phi^{-1}$. 

\begin{figure}[h]
{
   \fontsize{9pt}{11pt}\selectfont
   \def\svgwidth{1.8in}
   \begin{center}
   %% Creator: Inkscape 1.2.2 (b0a84865, 2022-12-01), www.inkscape.org
%% PDF/EPS/PS + LaTeX output extension by Johan Engelen, 2010
%% Accompanies image file 'pmlink.pdf' (pdf, eps, ps)
%%
%% To include the image in your LaTeX document, write
%%   \input{<filename>.pdf_tex}
%%  instead of
%%   \includegraphics{<filename>.pdf}
%% To scale the image, write
%%   \def\svgwidth{<desired width>}
%%   \input{<filename>.pdf_tex}
%%  instead of
%%   \includegraphics[width=<desired width>]{<filename>.pdf}
%%
%% Images with a different path to the parent latex file can
%% be accessed with the `import' package (which may need to be
%% installed) using
%%   \usepackage{import}
%% in the preamble, and then including the image with
%%   \import{<path to file>}{<filename>.pdf_tex}
%% Alternatively, one can specify
%%   \graphicspath{{<path to file>/}}
%% 
%% For more information, please see info/svg-inkscape on CTAN:
%%   http://tug.ctan.org/tex-archive/info/svg-inkscape
%%
\begingroup%
  \makeatletter%
  \providecommand\color[2][]{%
    \errmessage{(Inkscape) Color is used for the text in Inkscape, but the package 'color.sty' is not loaded}%
    \renewcommand\color[2][]{}%
  }%
  \providecommand\transparent[1]{%
    \errmessage{(Inkscape) Transparency is used (non-zero) for the text in Inkscape, but the package 'transparent.sty' is not loaded}%
    \renewcommand\transparent[1]{}%
  }%
  \providecommand\rotatebox[2]{#2}%
  \newcommand*\fsize{\dimexpr\f@size pt\relax}%
  \newcommand*\lineheight[1]{\fontsize{\fsize}{#1\fsize}\selectfont}%
  \ifx\svgwidth\undefined%
    \setlength{\unitlength}{186.56903711bp}%
    \ifx\svgscale\undefined%
      \relax%
    \else%
      \setlength{\unitlength}{\unitlength * \real{\svgscale}}%
    \fi%
  \else%
    \setlength{\unitlength}{\svgwidth}%
  \fi%
  \global\let\svgwidth\undefined%
  \global\let\svgscale\undefined%
  \makeatother%
  \begin{picture}(1,0.73324045)%
    \lineheight{1}%
    \setlength\tabcolsep{0pt}%
    \put(0,0){\includegraphics[width=\unitlength,page=1]{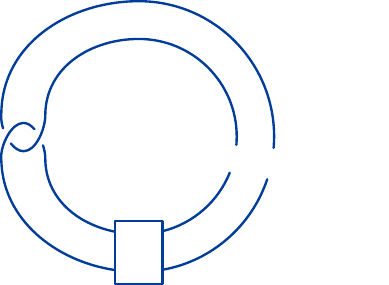}}%
    \put(0.3429329,0.05933123){\color[rgb]{0,0.23921569,0.6}\makebox(0,0)[lt]{\lineheight{1.25}\smash{\begin{tabular}[t]{l}$1$\end{tabular}}}}%
    \put(0.68187809,0.64398933){\color[rgb]{0,0.23921569,0.6}\makebox(0,0)[lt]{\lineheight{1.25}\smash{\begin{tabular}[t]{l}$(B,1)$\end{tabular}}}}%
    \put(0.32902905,0.36949257){\color[rgb]{0.76078431,0,0.12941176}\makebox(0,0)[lt]{\lineheight{1.25}\smash{\begin{tabular}[t]{l}$(R,1)$\end{tabular}}}}%
    \put(0,0){\includegraphics[width=\unitlength,page=2]{pmlink.pdf}}%
    \put(0.80877885,0.21530284){\color[rgb]{0.05098039,0.52156863,0}\makebox(0,0)[lt]{\lineheight{1.25}\smash{\begin{tabular}[t]{l}$(G,0)$\end{tabular}}}}%
  \end{picture}%
\endgroup%

   \end{center}
   \caption{}
   \label{pmlink}
}
\end{figure}

\end{example}

For the rest of the paper, we will mostly be concerned with $n$-surgery homeomorphisms (rather than those that change an $n$-surgery to a $(-n)$-surgery).

\subsection{$n$-slice knots from $n$-surgery homeomorphisms}
\label{nslicehomo}
Let $X$ be a smooth, closed, oriented 4-manifold, and let $X^{\circ}=X\backslash int(B^4)$. Let $K$ be a knot in $\partial X^{\circ}\cong S^3$. Suppose $K$ bounds a properly embedded disk $D$ in $X^{\circ}$. There exists a tubular neighborhood $\nu(D)\cong D^2\times D^2$, where $D$ is identified with $D^2\times \{0\}$. Pick a point $p \in \partial D^2$ and denote $S^1\times \{p\}$ by $K_D$. Following Section 2.2 of \cite{Kai}, we make the following definition.
\begin{definition}
The knot $K\subset \partial X^{\circ}$ is \emph{$n$-slice}, if $lk(K,K_D)=n$ in $\partial X^{\circ}$.
\end{definition}
Denote $X^{\circ}\backslash \nu(D)$ by $E(D)$ and denote the trace of the $n$-surgery along $K$ by $X_n(K)$. By the trace embedding lemma \cite[Lemma 3.3]{traceemb}, if $K$ is $n$-slice in $X$, then $-X_n(K)$ is smoothly embedded in $X$. In particular, we have that $[D]\cdot [D]=-n$. 

Now, given an $n$-surgery homeomorphism $\phi: S_n^3(K)\rightarrow S^3_n(J)$, we define 
$$X_{(D,\phi)}=-X_{n}(J)\cup_{\phi} E(D).$$

For $n=0$, if the disk $D$ is not null-homologous, then it is possible that $X_{(D,\phi)}$ is not homeomorphic to $X$. Note that if $X$ is definite, such as $\#^m\mathbb{CP}^2$, every $0$-slice disk is null-homologous. 

\begin{example}(\cite{cc} Example 5.3)
Let $X$ be $S^2\times S^2$. Since $S^2\times D^2 \cong X_0(U)$, $S^2\times D^2$ is the exterior of some disk in $X^{\circ}$. Let $\phi: S_0^3(U)\rightarrow S_0^3(U)$ be a homeomorphism that maps the $0$-framed meridian of $U$ to a $1$-framed meridian of $U$. Then $X_{(D,\phi)}\cong \mathbb{CP}^2 \# \overline{\mathbb{CP}^2}$.
\end{example}

However, for $X=\#^m\overline{\mathbb{CP}^2}$, we have that $X_{(D,\phi)}$ is homeomorphic to $X$.
\begin{proposition}\label{prop:definite}
Let $\phi:S_n^3(K)\rightarrow S_n^3(J)$ be an $n$-surgery homeomorphism and $D$ be an $n$-slice disk bounded by $K$. If $X$ is simply-connected and negative definite with $n\neq 0$, then $X_{(D,\phi)}$ is homeomorphic to $X$.
\end{proposition}
\begin{proof}
Denote the intersection pairing of a 4-manifold $M$ on $H_2(M;\mathbb{Q})$ by $Q_M$. 
Since $n\neq 0$, we have $H_2(S_n^3(K);\mathbb{Q}) = H_1(S_n^3(K);\mathbb{Q}) = 0$. Thus, by the Mayer–Vietoris sequence for $X$, $H_2(X;\mathbb{Q})\cong H_2(E(D);\mathbb{Q})\oplus H_2(-X_n(K);\mathbb{Q})$. Similarly, for $X_{(D,\phi)}$, we have $H_2(X_{(D,\phi)};\mathbb{Q})\cong H_2(E(D);\mathbb{Q})\oplus H_2(-X_n(J);\mathbb{Q})$.
Thus, $Q_X$ is isomorphic to $Q_{E(D)}\oplus Q_{-X_n(K)}$ over $\mathbb{Q}$, and $Q_{X_{(D,\phi)}}$ is isomorphic to $Q_{E(D)}\oplus Q_{-X_n(J)}$ over $\mathbb{Q}$. Since $Q_{X_n(K)}\cong Q_{X_n(J)}\cong (n)$ and $X$ is negative definite, we have that $X_{(D,\phi)}$ is also negative definite. By Donaldson's theorem, the intersection forms of $X$ and $X_{(D,\phi)}$ are diagonalizable over $\mathbb{Z}$. As in the proof of Lemma 3.3 in \cite{MP21} for the case $n=0$,  $X_{(D,\phi)}$ is simply-connected. Hence, by Freedman's theorem $X_{(D,\phi)}$ is homeomorphic to $X$.
\end{proof}

\begin{proposition}
The knot $J$ is $n$-slice in $X_{(D,\phi)}$.
\end{proposition}
\begin{proof}
The knot trace is $X_{-n}(-J)=B^4\cup_{(-J,-n)}\{\text{$2$-handle}\}$. Remove the $B^4$ from $X_{(D,\phi)}$, and the core of the 2-handle gives an $n$-slice disk of $J$. 
\end{proof}
\subsection{Extendability over $n$-traces}
For $(X, X_{(D,\phi)})$ to be a potential exotic pair, we need that the $n$-surgery homeomorphism $\phi: S_n^3(K)\rightarrow S_n^3(J)$ does not extend smoothly to an $n$-trace diffeomorphism $\Phi: X_n(K)\rightarrow X_n(J)$. In some cases, one can see that $\phi$ actually extends smoothly over $n$-traces.

\begin{example}\label{eg:trace}
Given an $|n|$-RBG link $L$ such that $(R,r)=(U,0)$, $B$ and $G$ are meridians of $R$ and $\psi_B$ (resp. $\psi_G$) is induced by sliding $B$ (resp. $G$) over $R$ and a slam-dunk. Replacing $(R,r)$ by a dotted circle and doing the same diagram calculus, we obtain a diffeomorphism from $X_n(K_B)$ to $X_n(K_G)$, extending $\phi_L$. Note that $L$ is an $n$-special RBG link with $(R,r) = (U, 0)$ (see Definition \ref{def:special}).
\end{example}

Generalizing Definition 3.12 in \cite{MP21}, we say that an $n$-surgery homeomorphism $\phi: S_n^3(K)\rightarrow S_n^3(J)$ has \emph{property U}, if there exists a choice of surgery diagrams of $S_n(K)$ and $S_n(J)$, such that $\phi$ sends a $0$-framed meridian of $K$ to a $0$-framed curve $\gamma$ which appears unknotted in the diagram of $S_n(J)$. 
\begin{theorem}
If $\phi$ has property U, then there exists a diffeomorphism $\Phi: X_n(K)\rightarrow X_n(J)$ with $\Phi|_{\partial}=\phi$.
\end{theorem}
\begin{proof}
This is a generalization of Theorem 3.13 in \cite{MP21}.
\end{proof}

For example, all links in Example \ref{eg:trace} have property U.

If the mapping class group of the $n$-surgeries is trivial, then if there exists some $\phi:S_n^3(K)\rightarrow S_n^3(J)$ which does not extend over the trace, then the traces $X_n(K)$ and $X_n(J)$ are not diffeomorphic. In general, it is hard to obstruct extensibility smoothly, but we have obstructions for extending homeomorphically over the traces.

\begin{proposition}\label{thm:mchase}
Let $\phi: S_n^3(K)\rightarrow S_n^3(J)$ be an $n$-surgery homeomorphism, with $n\neq 0$, which induces $f_*: H_1(S_n^3(K)) \rightarrow H_1(S_n^3(J))$.  Then, $f$ extends to a trace homeomorphism $\Phi: X_n(K)\rightarrow X_n(J)$ if and only if $f_*([\mu_K])=\pm[\mu_{J}]$.
\end{proposition}
\begin{proof}
If $f_*([\mu_K])=\pm[\mu_{J}]$, then we can lift $f_*$ to an isometry $\Lambda: H_2(X_n(K))\rightarrow H_2(X_n(J))$ such that the following diagram commutes. 
\begin{center}
\begin{tikzcd}[cramped, sep=scriptsize]
  0 \arrow[r]   & H_2(X_n(K))\arrow[r]\arrow[d, "\Lambda"] & H_2(X_n(K),S_n^3(K)) \arrow[r] &H_1(S^3_n(K))\arrow[r]\arrow[d,"f_*"] &0\\
  0 \arrow[r]   & H_2(X_n(J))\arrow[r]                                   & H_2(X_n(K),S_n^3(K))\arrow[u, "\Lambda^*"] \arrow[r] &H_1(S^3_n(K))\arrow[r] &0
 \end{tikzcd}
\end{center}

Since $n\neq 0$, the geometric obstruction $\theta(f,\Lambda)$ vanishes for any morphism $(f,\Lambda)$. Moreover, the Kirby-Siebenmann invariants $\Delta(X_n(K))\equiv \Delta(X_n(J))\equiv 0 \text{ (mod $2$)}$. Then the result follows from \cite{dd} Corollary 0.8 (i).
\end{proof}

\begin{remark}
There exist homeomorphisms of $n$-surgeries that do not map meridian to meridian. For instance, consider an $n$-special RBG link with $l=2$, $r=3$ (cf. Section \ref{nspecial}). One can chase the meridian within the link diagram and construct a homeomorphism $f$ which maps meridian $[\mu_{K_G}]$ to $3[\mu_{K_B}]$. By Proposition \ref{thm:mchase}, we have that $f$ is not extensible over the $n$-trace.
\end{remark}

\begin{proposition}
Let $n$ be an integer such that $\{l|l^2=1\}=\{\pm 1\}\subset \mathbb{Z}/n\mathbb{Z}$.  Every $n$-surgery homeomorphism extends over traces. 
\end{proposition}
\begin{proof}
Consider the linking form $Q$ of $M$. Since $n$-traces have intersection form $(n)$, we have $Q([\mu_K],[\mu_K])=Q([\mu_J],[\mu_J])=1/n $ (mod $\mathbb{Z}$). Let $f_*([\mu_K])=l[\mu_J]$. Since $Q$ is invariant under $f_*$, we have that $l^2=1$ (mod $n$). The result follows from Proposition \ref{thm:mchase}.
\end{proof}
For example, when $n=1,2,4,p^k$ or $2p^k$ with $p$ an odd prime, every $n$-surgery homeomorphism extends over the traces homeomorphically. For other $n$, it is possible that $\phi: S_n^3(K)\rightarrow S_n^3(J)$ does not map meridian to meridian homologically; in such a case, if we also have that the mapping class group of $S_n^3(K)$ is trivial, then we conclude that their $n$-traces $X_n(K)$ and $X_n(J)$ are not homeomorphic.

\section{$n$-special RBG links}
\label{nspecial}
{
\renewcommand{\thetheorem}{\ref{def:special}}
\begin{definition}
A link $L=R\cup B\cup G$, with framings $r,b,g\in \mathbb{Z}$ repectively and a linking matrix $M_L$, is called an \emph{$n$-special RBG link}, if 
\begin{itemize}
\item $b=g=0$,\item there exist link isotopies $R\cup B\cong R\cup \mu_R \cong R\cup G$,\item $n=-det(M_L)$.
\end{itemize}
\end{definition}
\addtocounter{theorem}{-1}
}

We obtain a knot diagram of $(K_G,f_g)$ as follows. Choose a diagram of an $n$-special RBG link $L$. Isotope $L$ such that $B=\mu_R$. Slide $G$ over $R$ such that $G$ does not intersect some meridian disk $\Delta_B$ of $R$ bounded by $B$. Cancel the pair $(B,R)$. Denote the image of $(G,g)$ by $(G',g')$. Similarly, we can obtain $(K_B, f_b)$.

\begin{remark}
Since there is only one orientation preserving homeomorphism of $S^3$ up to isotopy,  in the standard (empty) diagram of $S^3$, the framed knot $(K_G,f_g)$ is isotopic to $(G',g')$ and $(K_B,f_b)$ is isotopic to $(B',b')$.
\end{remark}

Notice that isotopies and slides do not change the determinant of the linking matrix, and a slam-dunk changes the sign of the linking matrix. Therefore, the homeomorphism $\phi_L$ induced by $L$ is an $n$-surgery homeomorphism.

Let $\psi_B$ ($\psi_G$) be the homeomorphism induced by a sequence of Kirby moves.
Note that an $n$-special RBG link, together with $\psi_B$ and $\psi_G$, is an $|n|$-RBG link.

Orient the link $L$ such that $lk(R,G)=lk(R,B)=1$ and denote $lk(B,G)$ by $l$. Then $$n=-detM_L=l(rl-2).$$
We can get arbitrary $n$ by setting $l=\pm 1$ and changing $r$.
\begin{example} \label{eg:m}
Consider a family of special RBG links with $l=-1$ in Figure \ref{6_2a}.

\begin{figure}[h]
{
   \fontsize{9pt}{11pt}\selectfont
   \def\svgwidth{1.8in}
   \begin{center}
   %% Creator: Inkscape 1.2 (dc2aeda, 2022-05-15), www.inkscape.org
%% PDF/EPS/PS + LaTeX output extension by Johan Engelen, 2010
%% Accompanies image file '6_2a.pdf' (pdf, eps, ps)
%%
%% To include the image in your LaTeX document, write
%%   \input{<filename>.pdf_tex}
%%  instead of
%%   \includegraphics{<filename>.pdf}
%% To scale the image, write
%%   \def\svgwidth{<desired width>}
%%   \input{<filename>.pdf_tex}
%%  instead of
%%   \includegraphics[width=<desired width>]{<filename>.pdf}
%%
%% Images with a different path to the parent latex file can
%% be accessed with the `import' package (which may need to be
%% installed) using
%%   \usepackage{import}
%% in the preamble, and then including the image with
%%   \import{<path to file>}{<filename>.pdf_tex}
%% Alternatively, one can specify
%%   \graphicspath{{<path to file>/}}
%% 
%% For more information, please see info/svg-inkscape on CTAN:
%%   http://tug.ctan.org/tex-archive/info/svg-inkscape
%%
\begingroup%
  \makeatletter%
  \providecommand\color[2][]{%
    \errmessage{(Inkscape) Color is used for the text in Inkscape, but the package 'color.sty' is not loaded}%
    \renewcommand\color[2][]{}%
  }%
  \providecommand\transparent[1]{%
    \errmessage{(Inkscape) Transparency is used (non-zero) for the text in Inkscape, but the package 'transparent.sty' is not loaded}%
    \renewcommand\transparent[1]{}%
  }%
  \providecommand\rotatebox[2]{#2}%
  \newcommand*\fsize{\dimexpr\f@size pt\relax}%
  \newcommand*\lineheight[1]{\fontsize{\fsize}{#1\fsize}\selectfont}%
  \ifx\svgwidth\undefined%
    \setlength{\unitlength}{230.93449999bp}%
    \ifx\svgscale\undefined%
      \relax%
    \else%
      \setlength{\unitlength}{\unitlength * \real{\svgscale}}%
    \fi%
  \else%
    \setlength{\unitlength}{\svgwidth}%
  \fi%
  \global\let\svgwidth\undefined%
  \global\let\svgscale\undefined%
  \makeatother%
  \begin{picture}(1,0.60497879)%
    \lineheight{1}%
    \setlength\tabcolsep{0pt}%
    \put(0,0){\includegraphics[width=\unitlength,page=1]{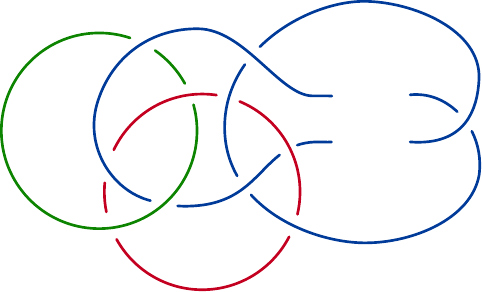}}%
    \put(0.73734436,0.34740447){\color[rgb]{0,0.23137255,0.6}\makebox(0,0)[lt]{\lineheight{1.25}\smash{\begin{tabular}[t]{l}$m$\end{tabular}}}}%
    \put(0,0){\includegraphics[width=\unitlength,page=2]{6_2a.pdf}}%
  \end{picture}%
\endgroup%

   \end{center}
   \caption{A one parameter family of $n$-special RBG links}
   \label{6_2a}
}
\end{figure}
The knot diagram for $K_B$ can be computed by sliding $B$ over $R$ and cancelling the pair $(G,R)$ (see Figure \ref{6_2}). Similarly, to get $K_G$, we isotope the link diagram such that $B$ becomes a circle, and let $\Delta_B$ be the inner domain bounded by $B$. Slide $G$ over $R$ along $\Delta_B$, and cancel the pair $(B,R)$.
\begin{figure}[h]
{
   \fontsize{9pt}{11pt}\selectfont
   \def\svgwidth{3.5in}
   \begin{center}
   %% Creator: Inkscape 1.2 (dc2aeda, 2022-05-15), www.inkscape.org
%% PDF/EPS/PS + LaTeX output extension by Johan Engelen, 2010
%% Accompanies image file 'K_B_1.pdf' (pdf, eps, ps)
%%
%% To include the image in your LaTeX document, write
%%   \input{<filename>.pdf_tex}
%%  instead of
%%   \includegraphics{<filename>.pdf}
%% To scale the image, write
%%   \def\svgwidth{<desired width>}
%%   \input{<filename>.pdf_tex}
%%  instead of
%%   \includegraphics[width=<desired width>]{<filename>.pdf}
%%
%% Images with a different path to the parent latex file can
%% be accessed with the `import' package (which may need to be
%% installed) using
%%   \usepackage{import}
%% in the preamble, and then including the image with
%%   \import{<path to file>}{<filename>.pdf_tex}
%% Alternatively, one can specify
%%   \graphicspath{{<path to file>/}}
%% 
%% For more information, please see info/svg-inkscape on CTAN:
%%   http://tug.ctan.org/tex-archive/info/svg-inkscape
%%
\begingroup%
  \makeatletter%
  \providecommand\color[2][]{%
    \errmessage{(Inkscape) Color is used for the text in Inkscape, but the package 'color.sty' is not loaded}%
    \renewcommand\color[2][]{}%
  }%
  \providecommand\transparent[1]{%
    \errmessage{(Inkscape) Transparency is used (non-zero) for the text in Inkscape, but the package 'transparent.sty' is not loaded}%
    \renewcommand\transparent[1]{}%
  }%
  \providecommand\rotatebox[2]{#2}%
  \newcommand*\fsize{\dimexpr\f@size pt\relax}%
  \newcommand*\lineheight[1]{\fontsize{\fsize}{#1\fsize}\selectfont}%
  \ifx\svgwidth\undefined%
    \setlength{\unitlength}{438.45197566bp}%
    \ifx\svgscale\undefined%
      \relax%
    \else%
      \setlength{\unitlength}{\unitlength * \real{\svgscale}}%
    \fi%
  \else%
    \setlength{\unitlength}{\svgwidth}%
  \fi%
  \global\let\svgwidth\undefined%
  \global\let\svgscale\undefined%
  \makeatother%
  \begin{picture}(1,0.33295338)%
    \lineheight{1}%
    \setlength\tabcolsep{0pt}%
    \put(0,0){\includegraphics[width=\unitlength,page=1]{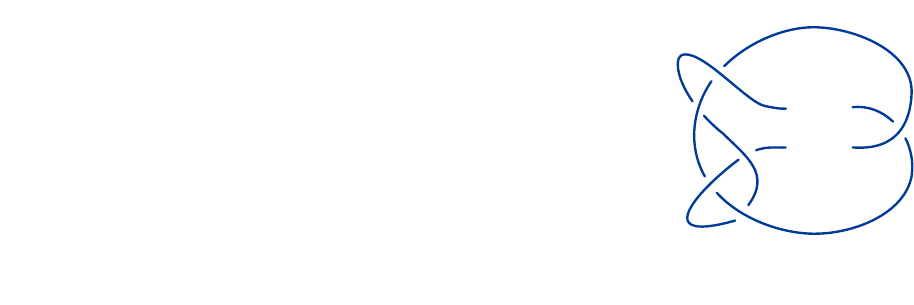}}%
    \put(0.88168846,0.18840072){\color[rgb]{0,0.23137255,0.6}\makebox(0,0)[lt]{\lineheight{1.25}\smash{\begin{tabular}[t]{l}$m$\end{tabular}}}}%
    \put(0,0){\includegraphics[width=\unitlength,page=2]{K_B_1.pdf}}%
    \put(0.33800404,0.21486179){\color[rgb]{0,0.23137255,0.6}\makebox(0,0)[lt]{\lineheight{1.25}\smash{\begin{tabular}[t]{l}$m$\end{tabular}}}}%
    \put(0,0){\includegraphics[width=\unitlength,page=3]{K_B_1.pdf}}%
    \put(0.18733917,0.06027502){\color[rgb]{0,0,0}\makebox(0,0)[lt]{\lineheight{1.25}\smash{\begin{tabular}[t]{l}$r$\end{tabular}}}}%
    \put(0,0){\includegraphics[width=\unitlength,page=4]{K_B_1.pdf}}%
    \put(0.5145166,0.21453447){\color[rgb]{0,0,0}\makebox(0,0)[lt]{\lineheight{1.25}\smash{\begin{tabular}[t]{l}slam-dunk\end{tabular}}}}%
    \put(0.84640306,0.00361824){\color[rgb]{0,0,0}\makebox(0,0)[lt]{\lineheight{1.25}\smash{\begin{tabular}[t]{l}$K_B$\end{tabular}}}}%
    \put(0.1246023,0.00361824){\color[rgb]{0,0,0}\makebox(0,0)[lt]{\lineheight{1.25}\smash{\begin{tabular}[t]{l}slide $B$ over $R$ \end{tabular}}}}%
  \end{picture}%
\endgroup%

   \end{center}
   \caption{Obtaining the diagram of $K_B$ with one slide}
   \label{6_2}
}
\end{figure}
\begin{figure}[h]
{
   \fontsize{9pt}{11pt}\selectfont
   \def\svgwidth{2in}
   \begin{center}
   %% Creator: Inkscape 1.2 (dc2aeda, 2022-05-15), www.inkscape.org
%% PDF/EPS/PS + LaTeX output extension by Johan Engelen, 2010
%% Accompanies image file 'K_Gmr.pdf' (pdf, eps, ps)
%%
%% To include the image in your LaTeX document, write
%%   \input{<filename>.pdf_tex}
%%  instead of
%%   \includegraphics{<filename>.pdf}
%% To scale the image, write
%%   \def\svgwidth{<desired width>}
%%   \input{<filename>.pdf_tex}
%%  instead of
%%   \includegraphics[width=<desired width>]{<filename>.pdf}
%%
%% Images with a different path to the parent latex file can
%% be accessed with the `import' package (which may need to be
%% installed) using
%%   \usepackage{import}
%% in the preamble, and then including the image with
%%   \import{<path to file>}{<filename>.pdf_tex}
%% Alternatively, one can specify
%%   \graphicspath{{<path to file>/}}
%% 
%% For more information, please see info/svg-inkscape on CTAN:
%%   http://tug.ctan.org/tex-archive/info/svg-inkscape
%%
\begingroup%
  \makeatletter%
  \providecommand\color[2][]{%
    \errmessage{(Inkscape) Color is used for the text in Inkscape, but the package 'color.sty' is not loaded}%
    \renewcommand\color[2][]{}%
  }%
  \providecommand\transparent[1]{%
    \errmessage{(Inkscape) Transparency is used (non-zero) for the text in Inkscape, but the package 'transparent.sty' is not loaded}%
    \renewcommand\transparent[1]{}%
  }%
  \providecommand\rotatebox[2]{#2}%
  \newcommand*\fsize{\dimexpr\f@size pt\relax}%
  \newcommand*\lineheight[1]{\fontsize{\fsize}{#1\fsize}\selectfont}%
  \ifx\svgwidth\undefined%
    \setlength{\unitlength}{371.34285274bp}%
    \ifx\svgscale\undefined%
      \relax%
    \else%
      \setlength{\unitlength}{\unitlength * \real{\svgscale}}%
    \fi%
  \else%
    \setlength{\unitlength}{\svgwidth}%
  \fi%
  \global\let\svgwidth\undefined%
  \global\let\svgscale\undefined%
  \makeatother%
  \begin{picture}(1,0.67216621)%
    \lineheight{1}%
    \setlength\tabcolsep{0pt}%
    \put(0,0){\includegraphics[width=\unitlength,page=1]{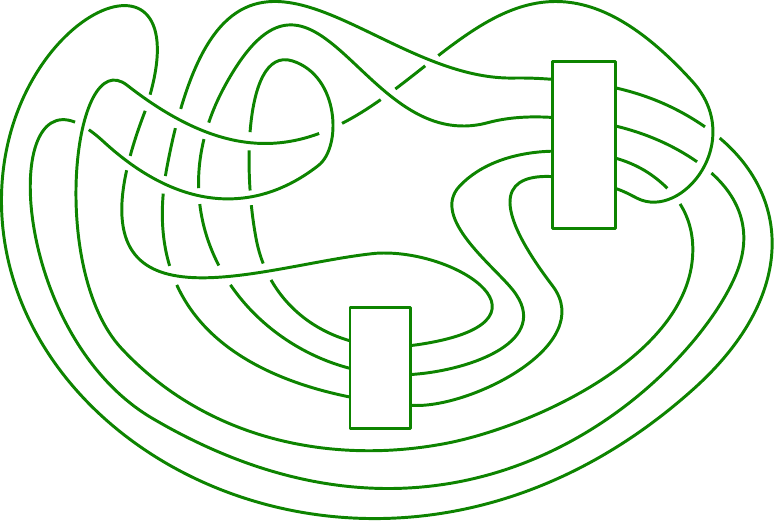}}%
    \put(0.48057876,0.1890843){\color[rgb]{0.05098039,0.51764706,0}\makebox(0,0)[lt]{\lineheight{1.25}\smash{\begin{tabular}[t]{l}$r$\end{tabular}}}}%
    \put(0.73552498,0.48522394){\color[rgb]{0.05098039,0.51764706,0}\makebox(0,0)[lt]{\lineheight{1.25}\smash{\begin{tabular}[t]{l}$m$\end{tabular}}}}%
  \end{picture}%
\endgroup%

   \end{center}
   \caption{The diagram of $K_G$}
}
\end{figure}

If $m=1$, $K_B$ is the figure-eight knot, and $K_G$ is also the figure-eight knot for any $r$.
For $r=-1, m=0$, we have $S_1^3(K_B)\cong S_1^3(K_G)$. Using SnapPy \cite{SnapPy}, we identify $K_B$ as $6_2$, which is the mirror of $6_2$ in Rolfsen knot table, and $K_G$ as $K13n3596$.  This gives an example of small knots which have the same $1$-surgery.
\end{example}

\begin{remark}
By Theorem 3.7 in \cite{bb}, for a knot $K$ with an annulus presentation, one can construct another knot $K^{\prime}$ via the $(*n)$ operation, such that $S_n^3(K)\cong S_n^3(K^{\prime})$. 

In Example \ref{eg:m} with $m=0$, the knot $K_B$ has an annulus presentation as shown in Figure \ref{annulus}. The knot $K^{\prime}$ obtained by applying ($*(-1)$) operation on the mirror of $K_B$, is the mirror of the knot $K_G$ in Example \ref{eg:m}.  

\begin{figure}[h]
{
   \fontsize{9pt}{11pt}\selectfont
   \def\svgwidth{1.3in}
   \begin{center}
   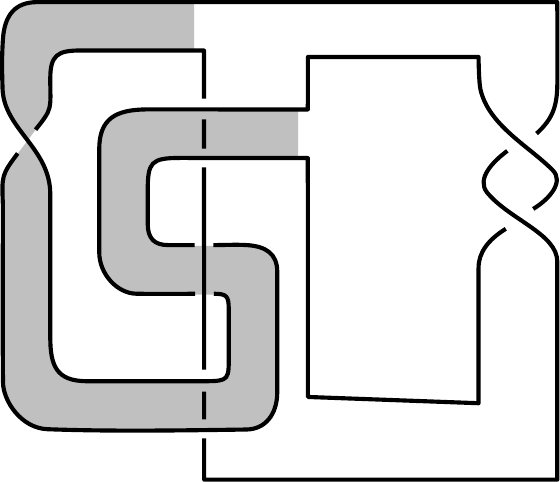
   \end{center}
   \caption{An annulus presentation of $K_B$ in Figure \ref{6_2} with $m=0$}
   \label{annulus}
}
\end{figure}

\end{remark}

There are, however, some cases where an $n$-special RBG link produces identical knots.

\begin{proposition}\label{prop:R}
Let $L=\{(R,r),(B,0),(G,0)\}$ be an $n$-special RBG link. If $(R,r)=(U,0)$ and $R$ bounds a disk $\Delta_R$ such that $|\Delta_R\cap B|=|\Delta_R\cap G|=1$, then $K_B=K_G$.
\end{proposition}

\begin{proof}
Slide $B$ over $G$ such that the resulting blue component $B^{\prime}$ does not intersect with $\Delta_R$. Since $R$ has framing $0$, we can cancel the pair $(R,G)$. Since these induce a homeomorphism from $S^3_{0,0}(R,G)$ to $S^3$, the knot $K_B$ is isotopic to $B^{\prime}$. Similarly, the knot $K_G$ is obtained by sliding $G$ over $B$ using the same band as above and therefore has the same diagram as $K_B$.
\end{proof}

\begin{proposition}
Let $L$ be an $n$-special RBG link. If $B$ bounds a properly embedded disk $\Delta_B$ such that $|\Delta_B\cap R|=1$, $|\Delta_B\cap G|<2$, and if $G$ bounds a properly embedded disk $\Delta_G$ such that $|\Delta_G\cap R|=1$, $|\Delta_G\cap B|\le 2$, then $K_B=K_G$. (All intersections are required to be transverse.)
\end{proposition}
\begin{proof}
This is a generalization of Proposition 4.11 in \cite{MP21}, which was for $n=0$. The proof in \cite{MP21} is independent of the framings of the RBG link.
\end{proof}
\begin{remark}
From Example \ref{eg:m}, we see that there exists a special RBG link with disks $\Delta_G$, $\Delta_B$, such that $|\Delta_G\cap R|=1$, $|\Delta_B\cap R|=1$, and $|\Delta_G\cap B|=1$, $|\Delta_B\cap G|=3$, but the associated knots $K_B$, $K_G$ are not isotopic.
\end{remark}

\begin{example}
Consider a family of special RBG links with four twisting boxes as in Figure \ref{little_monster}. Since the linking number $l$ between $B$ and $G$ is $-1$, if $r=1$, then $n=3$. Therefore, we obtain a family of $3$-special RBG links parametrized by the numbers of twists $(a, b, c, d)$. For each choice of $(a, b, c, d)$, we denote the green knot associated to the link by $K_G(a,b,c,d)$, and the corresponding blue knot by $K_B(a,b,c,d)$.
\begin{figure}[h]
{
   \fontsize{9pt}{11pt}\selectfont
   \def\svgwidth{3.2in}
   \begin{center}
   %% Creator: Inkscape 1.2.2 (b0a84865, 2022-12-01), www.inkscape.org
%% PDF/EPS/PS + LaTeX output extension by Johan Engelen, 2010
%% Accompanies image file 'little_green.pdf' (pdf, eps, ps)
%%
%% To include the image in your LaTeX document, write
%%   \input{<filename>.pdf_tex}
%%  instead of
%%   \includegraphics{<filename>.pdf}
%% To scale the image, write
%%   \def\svgwidth{<desired width>}
%%   \input{<filename>.pdf_tex}
%%  instead of
%%   \includegraphics[width=<desired width>]{<filename>.pdf}
%%
%% Images with a different path to the parent latex file can
%% be accessed with the `import' package (which may need to be
%% installed) using
%%   \usepackage{import}
%% in the preamble, and then including the image with
%%   \import{<path to file>}{<filename>.pdf_tex}
%% Alternatively, one can specify
%%   \graphicspath{{<path to file>/}}
%% 
%% For more information, please see info/svg-inkscape on CTAN:
%%   http://tug.ctan.org/tex-archive/info/svg-inkscape
%%
\begingroup%
  \makeatletter%
  \providecommand\color[2][]{%
    \errmessage{(Inkscape) Color is used for the text in Inkscape, but the package 'color.sty' is not loaded}%
    \renewcommand\color[2][]{}%
  }%
  \providecommand\transparent[1]{%
    \errmessage{(Inkscape) Transparency is used (non-zero) for the text in Inkscape, but the package 'transparent.sty' is not loaded}%
    \renewcommand\transparent[1]{}%
  }%
  \providecommand\rotatebox[2]{#2}%
  \newcommand*\fsize{\dimexpr\f@size pt\relax}%
  \newcommand*\lineheight[1]{\fontsize{\fsize}{#1\fsize}\selectfont}%
  \ifx\svgwidth\undefined%
    \setlength{\unitlength}{387.50062212bp}%
    \ifx\svgscale\undefined%
      \relax%
    \else%
      \setlength{\unitlength}{\unitlength * \real{\svgscale}}%
    \fi%
  \else%
    \setlength{\unitlength}{\svgwidth}%
  \fi%
  \global\let\svgwidth\undefined%
  \global\let\svgscale\undefined%
  \makeatother%
  \begin{picture}(1,0.76404563)%
    \lineheight{1}%
    \setlength\tabcolsep{0pt}%
    \put(0,0){\includegraphics[width=\unitlength,page=1]{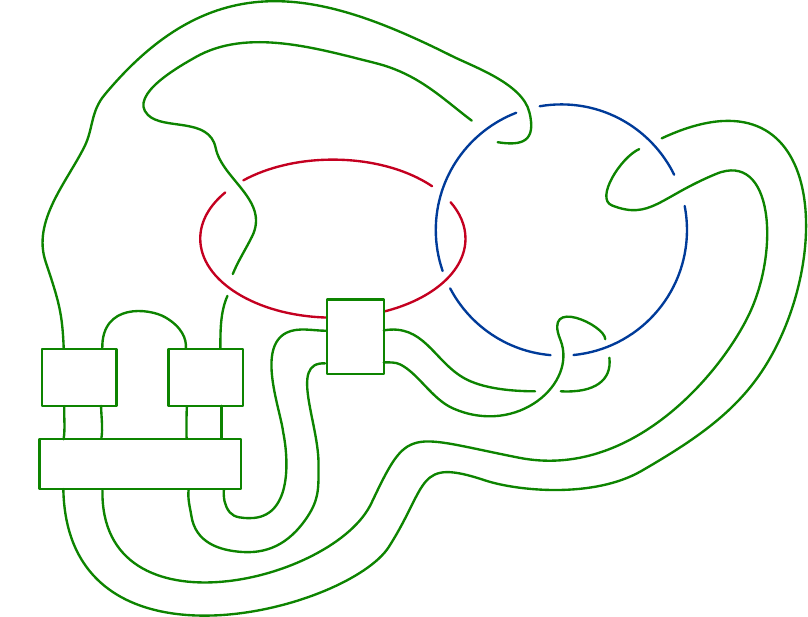}}%
    \put(-0.00029477,0.50139599){\color[rgb]{0.05098039,0.51764706,0}\makebox(0,0)[lt]{\lineheight{1.25}\smash{\begin{tabular}[t]{l}$0$\end{tabular}}}}%
    \put(0.70507609,0.66224417){\color[rgb]{0,0.23137255,0.6}\makebox(0,0)[lt]{\lineheight{1.25}\smash{\begin{tabular}[t]{l}$0$\end{tabular}}}}%
    \put(0.0920285,0.28443218){\color[rgb]{0.05098039,0.51764706,0}\makebox(0,0)[lt]{\lineheight{1.25}\smash{\begin{tabular}[t]{l}$a$\end{tabular}}}}%
    \put(0.24818471,0.28224301){\color[rgb]{0.05098039,0.51764706,0}\makebox(0,0)[lt]{\lineheight{1.25}\smash{\begin{tabular}[t]{l}$b$\end{tabular}}}}%
    \put(0.16264943,0.18030719){\color[rgb]{0.05098039,0.51764706,0}\makebox(0,0)[lt]{\lineheight{1.25}\smash{\begin{tabular}[t]{l}$c$\end{tabular}}}}%
    \put(0.43038763,0.33668578){\color[rgb]{0.05098039,0.51764706,0}\makebox(0,0)[lt]{\lineheight{1.25}\smash{\begin{tabular}[t]{l}$d$\end{tabular}}}}%
    \put(0.41284174,0.58852801){\color[rgb]{0.76862745,0,0.1254902}\makebox(0,0)[lt]{\lineheight{1.25}\smash{\begin{tabular}[t]{l}$1$\end{tabular}}}}%
  \end{picture}%
\endgroup%

   \end{center}
   \caption{}
   \label{little_monster}
}
\end{figure}

For example, let $(a,b,c,d)=(-2,1,-1,-1)$. The knot $K_G(-2,2,-1,-1)$ is the non-hyperbolic knot T$(-2, 3)\#$T$(2, 5)$, and the knot $K_B(-2,2,-1,-1)$ is recognized as $K12n121$ by SnapPy \cite{SnapPy} (see Figure \ref{torusfish}). Since $K_G(-2,2,-1,-1)$ and $K_B(-2,2,-1,-1)$ are generated by a $3$-special RBG link, they have the same $3$-surgery. Thus, $3$ is a non-characterizing slope for T$(-2, 3)\#$T$(2, 5)$, recovering an example in \cite{torusknot}.

\begin{figure}[h]
{
   \fontsize{9pt}{11pt}\selectfont
   \def\svgwidth{3.2in}
   \begin{center}
   %% Creator: Inkscape 1.2.2 (b0a84865, 2022-12-01), www.inkscape.org
%% PDF/EPS/PS + LaTeX output extension by Johan Engelen, 2010
%% Accompanies image file 'torusfish.pdf' (pdf, eps, ps)
%%
%% To include the image in your LaTeX document, write
%%   \input{<filename>.pdf_tex}
%%  instead of
%%   \includegraphics{<filename>.pdf}
%% To scale the image, write
%%   \def\svgwidth{<desired width>}
%%   \input{<filename>.pdf_tex}
%%  instead of
%%   \includegraphics[width=<desired width>]{<filename>.pdf}
%%
%% Images with a different path to the parent latex file can
%% be accessed with the `import' package (which may need to be
%% installed) using
%%   \usepackage{import}
%% in the preamble, and then including the image with
%%   \import{<path to file>}{<filename>.pdf_tex}
%% Alternatively, one can specify
%%   \graphicspath{{<path to file>/}}
%% 
%% For more information, please see info/svg-inkscape on CTAN:
%%   http://tug.ctan.org/tex-archive/info/svg-inkscape
%%
\begingroup%
  \makeatletter%
  \providecommand\color[2][]{%
    \errmessage{(Inkscape) Color is used for the text in Inkscape, but the package 'color.sty' is not loaded}%
    \renewcommand\color[2][]{}%
  }%
  \providecommand\transparent[1]{%
    \errmessage{(Inkscape) Transparency is used (non-zero) for the text in Inkscape, but the package 'transparent.sty' is not loaded}%
    \renewcommand\transparent[1]{}%
  }%
  \providecommand\rotatebox[2]{#2}%
  \newcommand*\fsize{\dimexpr\f@size pt\relax}%
  \newcommand*\lineheight[1]{\fontsize{\fsize}{#1\fsize}\selectfont}%
  \ifx\svgwidth\undefined%
    \setlength{\unitlength}{509.93354227bp}%
    \ifx\svgscale\undefined%
      \relax%
    \else%
      \setlength{\unitlength}{\unitlength * \real{\svgscale}}%
    \fi%
  \else%
    \setlength{\unitlength}{\svgwidth}%
  \fi%
  \global\let\svgwidth\undefined%
  \global\let\svgscale\undefined%
  \makeatother%
  \begin{picture}(1,0.51216384)%
    \lineheight{1}%
    \setlength\tabcolsep{0pt}%
    \put(0,0){\includegraphics[width=\unitlength,page=1]{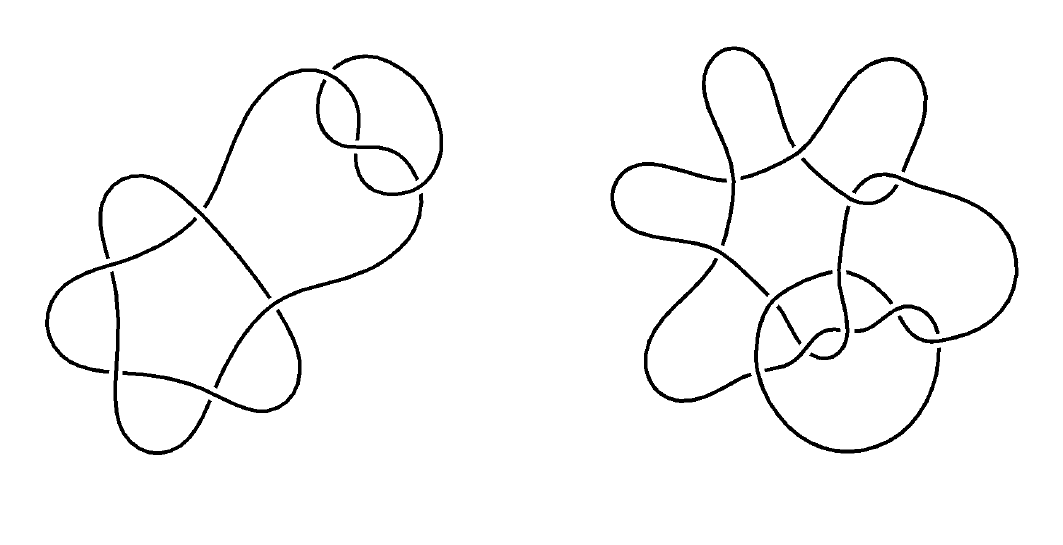}}%
    \put(0.09299208,0.00466233){\color[rgb]{0,0,0}\makebox(0,0)[lt]{\lineheight{1.25}\smash{\begin{tabular}[t]{l}$K_G(-2,2,-1,-1)$\end{tabular}}}}%
    \put(0.64864462,0.00466233){\color[rgb]{0,0,0}\makebox(0,0)[lt]{\lineheight{1.25}\smash{\begin{tabular}[t]{l}$K_B(-2,2,-1,-1)$\end{tabular}}}}%
  \end{picture}%
\endgroup%

   \end{center}
   \caption{}
   \label{torusfish}
}
\end{figure}

This family also produces knot pairs with small crossing numbers (summing up to $20$). When $(a,b,c,d)=(1,2,-1,-1)$, we have that $S^3_3(6_3)\cong S^3_3(K14n15962)$. For $(a,b,c,d)=(0,1,-1,-1)$, we obtain that $S^3_3(6_2)\cong S^3_3(K14n10164)$, which can also be obtained from Example \ref{eg:m} with $r=1, m=0$. Finally, when $(a,b,c,d)=(-1,2,-1,-1)$, the link generates the knot pair $(K_G, K_B) = (10_{125}, 10_{132})$, whose $3$-surgeries are isometric hyperbolic $3$-manifolds. All of the knot pairs in this example have the property that $s(K_B) = s(K_G)$.
\end{example}

\section{potential exotica}
\label{nsinv}

In this section, we follow the recipe given in Section \ref{nslicehomo} and look for exotica using $n$-special links and the Rasmussen's $s$-invariant. More specifically, we can use $n$-special RBG links to find knot pairs $(K,J)$ that share the same $n$-surgery. If $K$ is $n$-slice in $\#^m \overline{\mathbb{CP}^2}$ and $J$ is not $n$-slice in $\#^m \overline{\mathbb{CP}^2}$ for some $m$, then we can build an exotic negative-definite $4$-manifold (cf. Theorem \ref{prop:definite}). We first discuss the use of the $s$-invariant in obstructing a knot from being $n$-slice and then present an example which generates potential exotica.

\subsection{$n$-special RBG link and the $s$-invariant}

We generalize the results of Nakamura \cite{Kai} to $n$-special RBG links.

\begin{proposition}[Analogue of Lemma 3.1 in \cite{Kai}]\label{0}
Let $L=\{(R,r),(B,0),(G,0)\}$ be an $n$-special RBG link. If $R$ is $r$-slice in a smooth oriented closed $4$-manifold $W$, then $$X_n(K_B)\# W\cong X_n(K_G)\# W.$$
\end{proposition}
\begin{proof}
Let $Z$ be the 4-manifold obtained by attaching two $2$-handles to $W^{\circ}\backslash \nu(D)$ along $\{(B,0),(G,0)\}$, where $D$ is an $r$-slice disk of $R$. Since $(B,0)$ is isotopic to $(\mu_R,0)$, the $2$-handle attached along $(B,0)$ fills $\nu(D)$. Slide $G$ over $R$ such that $G$ does not intersect some meridian disk $\Delta_B$ bounded by $B$. This changes $(G,0)$ to $(K_G,n)$ and induces $Z\cong W^{\circ} \cup_{(K_G,n)} 2h \cong X_n(K_G) \# W$. Similarly, $Z \cong X_n(K_B) \# W$, and therefore $X_n(K_B)\# W\cong X_n(K_G)\# W$.
\end{proof}

\begin{corollary}[Analogue of Corollary 3.2 in \cite{Kai}]\label{1}
Let $L=\{(R,r),(B,0),(G,0)\}$ be an $n$-special RBG link, such that $R$ is $r$-slice in $W$. If $K_B$ is $n$-slice in $X$, then $K_G$ is $n$-slice in $X\# -W$.
\end{corollary}
\begin{proof}
If $K_B$ is $n$-slice in $X$, then $-X_n(K_B)$ smoothly embeds in $X$, which implies that $-X_n(K_B)\# -W$ smoothly embeds in $X\#-W$. By Proposition \ref{0}, $ -X_n(K_G)\#-W$ also smoothly embeds in $X\#-W$. Therefore, $K_G$ is $n$-slice in $X\#-W$.
\end{proof}

\begin{proposition} [Analogue of Lemma 3.11 in \cite{Kai}]\label{2}
Let $L=\{(R,r),(B,0),(G,0)\}$ be an $n$-special RBG link. If $R$ is $(r-1)$-slice in some closed 4-manifold $W$, then $K_B,K_G$ are $(n+1)$-slice in $W\#\overline{\mathbb{CP}^2}$. 
\end{proposition}
\begin{proof}
We use $X_{f_K,f_J}(K,J)$ to denote the 4-manifold obtained by attaching two $2$-handles along the framed link $\{(K,f_K),(J,f_J)\}$ to a 4-ball. 

Consider the slides of $B$ over $R$ which are applied to $L$ in obtaining the knot $K_B$. These induce a diffeomorphism between $X_{r,n+1}(R,K_B)$ and $X_{r,1}(R,B)$. Next, Slide $R$ over $B$ to separate $B$ and $R$, and obtain a diffeomorphism $X_{r,1}(R,B)\cong X_{r-1}(R)\# {\mathbb{CP}}^2$. Since $R$ is $(r-1)$-slice in $W$, the trace $-X_{r-1}(R)$ smoothly embeds in $W$, which implies that $-X_{r-1}(R)\# \overline{\mathbb{CP}^2}$ smoothly embeds in $W\# \overline{\mathbb{CP}^2}$. Since $-X_{n+1}(K_B)\subset -X_{r,n+1}(R,K_B)$, the trace $-X_{n+1}(K_B)$  smoothly embeds in $W\# \overline{\mathbb{CP}^2}$, i.e. $K_B$ is $(n+1)$-slice in $W\#\overline{\mathbb{CP}^2}$.
\end{proof}

\begin{comment}
\begin{theorem}
Let $L=\{(R,r),(B,0),(G,0)\}$ be a $(-k)$-special RBG link.
\begin{enumerate}[(a)]
\item  If $R$ is $r$-slice in some $\#^m \overline{\mathbb{CP}^2}$ and $K_B$ is $(-k)$-slice in some $\#^l {\mathbb{CP}^2}$, then $K_G$ is $(-k)$-slice in $\#^{l+m}\mathbb{CP}^2$.

\item If $R$ is $(r+1)$-slice in some $\#^m {\mathbb{CP}^2}$, then $K_B,K_G$ are $(-k-1)$-slice in $\#^{m+1}\mathbb{CP}^2$.
\end{enumerate}
\end{theorem}
\end{comment}
From now on, we assume that $n$ is a nonnegative integer.
\begin{theorem}
\label{thm:slice}
Let $L=\{(R,r),(B,0),(G,0)\}$ be an $n$-special RBG link.
\begin{enumerate}[(a)]
\item  If $R$ is $r$-slice in some $\#^m {\mathbb{CP}^2}$ and $K_B$ is $n$-slice in some $\#^l \overline{\mathbb{CP}^2}$, then $K_G$ is $n$-slice in $\#^{l+m} \overline{\mathbb{CP}^2}$.

\item If $R$ is $(r-1)$-slice in some $\#^m \overline{\mathbb{CP}^2}$, then $K_B,K_G$ are $(n+1)$-slice in $\#^{m+1}\overline{\mathbb{CP}^2}$.
\end{enumerate}
\end{theorem}

\begin{proof}
(a) If $R$ is $r$-slice in some $\#^m \mathbb{CP}^2$, and $K_B$ is $n$-slice in some $\#^l \overline{\mathbb{CP}^2}$, then by Corollary \ref{1}, $K_G$ is $n$-slice in $\#^{l+m} \overline{\mathbb{CP}^2}$. (b) If $R$ is $(r-1)$-slice in some $\#^m \overline{\mathbb{CP}^2}$, then by Proposition \ref{2}, $K_B,K_G$ are $(n+1)$-slice in $\#^{m+1} \overline{\mathbb{CP}^2}$.
\end{proof}
For obstructing $n$-sliceness in $\# ^l\overline{\mathbb{CP}^2}$, we will use the following adjunction inequality for the $s$-invariant in \cite{Fpq}.
\begin{theorem}[Corollary 1.4 in \cite{Fpq}]
\label{conj:s}
Let $W = \# ^l\overline{\mathbb{CP}^2}$. Let $K \subset \partial W^{\circ} = S^3$ be a knot, and $\Sigma\subset W^{\circ}$ a properly, smoothly embedded oriented surface with no closed components, such that $\partial\Sigma = K$. Then:
$$s(K)\le 2g(\Sigma)-|[\Sigma]|-[\Sigma]\cdot[\Sigma]$$
\end{theorem}
\begin{proposition}\label{Kai:s}
If $K$ is $n$-slice in some $\# ^l\overline{\mathbb{CP}^2}$, then $s(K)\le n-\sqrt{n}$. If $J$ is $(-n)$-slice in some $\# ^l{\mathbb{CP}^2}$, then $s(J)\ge -n+\sqrt{n}$.
\end{proposition}
\begin{proof}
Let $\{e_1,\cdots,e_l\}$ be an orthonormal basis of $H_2(\#^l\overline{\mathbb{CP}^2})$, and $[D]=a_1e_1+a_2e_2+\cdots+a_le_l$ in $H_2(\#^l\overline{\mathbb{CP}^2})$. Since $K$ is $n$-slice, $-[D]\cdot[D]=a_1^2+\cdots+a_l^2=n$. Since $|[D]|^2=(|a_1|+\cdots+|a_l|)^2\ge a_1^2+\cdots+a_l^2=n$, we have $s(K)\le n-\sqrt{n}$ by Theorem \ref{conj:s}. For the second part, consider the mirror of $J$, which is $n$-slice in $\# ^l\overline{\mathbb{CP}^2}$, and obtain that $s(J)\ge -n+\sqrt{n}$.
\end{proof}
Note that a similar inequality holds for the $\tau$-invariant (see Corollary 2.13 \cite{Kai}).
{
\renewcommand{\thetheorem}{\ref{thm:sinv}}
\begin{theorem}
Let $L=\{(R,r),(B,0),(G,0)\}$ be an $n$-special RBG link. 
\begin{enumerate}[(a)]
\item  If $R$ is $r$-slice in some $\#^m {\mathbb{CP}^2}$ and $K_B$ is $n$-slice in some $\#^l \overline{\mathbb{CP}^2}$, then $s(K_G)\le n-\sqrt{n}$.

\item If $R$ is $(r-1)$-slice in some $\#^m \overline{\mathbb{CP}^2}$, then $s(K_G)\le n+1-\sqrt{n+1}$.
\end{enumerate}
\end{theorem}
\addtocounter{theorem}{-1}
}
\begin{proof}
If $R$ is $r$-slice in some $\#^m \mathbb{CP}^2$ and $K_B$ is $n$-slice in some $\#^l \overline{\mathbb{CP}^2}$, then by Corollary \ref{1}, $K_G$ is $n$-slice in $\#^{l+m} \overline{\mathbb{CP}^2}$, then by Proposition \ref{Kai:s}, $s(K_G)\le n-\sqrt{n}$.

If $R$ is $(r-1)$-slice in some $\#^m \overline{\mathbb{CP}^2}$, then by Proposition \ref{2}, $K_G$ is $(n+1)$-slice in $\#^{m+1} \overline{\mathbb{CP}^2}$. Then, by Proposition \ref{Kai:s}, $s(K_G)\le n+1-\sqrt{n+1}$.
\end{proof}

Given an $n$-special RBG link $L$, one can modify the link to $L^{\prime}$ by decreasing the framing $r$ of $R$. Notice that the diagrams of $K_G$ can be obtained by adding full twists on parallel strands of the diagram of $K_G^{\prime}$ of the modified link. We will use the following lemma, which generalize Proposition 7.6 (8) in \cite{ff}, to get more precise bounds of $s(K_G)$.

\begin{lemma}\label{lem:s-l}
Let $K$ be a knot. Let $K^{\prime}$ be obtained by adding one positive twist to parallel strands of $K$. If the algebraic intersection between the parallel stands and the twisting disk $\Delta_T$ is $l$, then we have {$s(K^{\prime})\le s(K)-|l|+l^2$}.
\end{lemma}

\begin{proof}
Let $T$ be the boundary of $\Delta_T$. We build $W=\overline{\mathbb{CP}^2} \setminus (\interior{B^4} \cup \interior{B^4})$ by attaching a $(-1)$-framed $2$-handle along $T$ to $S^3\times \{1\}\subset S^3\times I$. Consider the annulus $A=K\times I\subset S^3\times I$. Choose a point $p\in K$ which is not on the twisting disk $\Delta_T$. Pick a tubular neighborhood $N$ of $\{p\}\times I$ in $W$, such that $N\cap (\Delta_T\times \{1\})=\varnothing$. Let $(D,X)$ be $(A\setminus N, W\setminus N)\cong (D^2, \overline{\mathbb{CP}^2} \setminus \interior{B^4})$. The boundary of $D$ is $-K\# K^{\prime}$. Since $K$ intersect $\Delta$ algebraically {$l$-times}, $D$ represents the class {$ l[\overline{\mathbb{CP}^1}]$} in $H^2(\overline{\mathbb{CP}^2})$. Applying Theorem \ref{conj:s} to $(D,X)$, we have 
$s(-K\# K^{\prime})\le g(D)-|[D]|-[D]\cdot[D]=${$-|l|+l^2$}. Thus, $s(K^{\prime})\le s(K)-|l|+l^2$.
\end{proof}

\begin{proposition}\label{prop:egR}
Let $L=\{(R,r),(B,0),(G,0)\}$ be an $n$-special RBG link, such that $R=U$, $r\ge 0$. Suppose there exists a disk $\Delta_G$, such that $\partial\Delta_G=G$, and $B,R$ intersect $\Delta_G$ geometrically once. Suppose also that there exists a disk $\Delta_R$, such that $\partial\Delta_R=R$, and $B,G$ intersect $\Delta_R$ geometrically once. If $K_B$ is $n$-slice in some $\#^l\overline{\mathbb{CP}^2}$, then $s(K_G)\le n-\sqrt{n}$.
\end{proposition}

\begin{proof}
Since there exists a disk $\Delta_G$, such that $\partial\Delta_G=G$, and $B,R$ intersect $\Delta_G$ geometrically once, the knot $K_B$ can be obtained by sliding $B$ over $R$ once, and therefore the diagram of $K_B$ does not depend on $r$. Let $L^0=\{(U,0),(B,0),(G,0)\}$. 
Proposition \ref{prop:R} implies that $K_B^0=K_G^0$. 

Observe that $K_G$ can be obtained from $K_G^0$ by adding an $r$-positive twist box on parallel strands. Since $|\Delta_G\cap B|=1$ implies $|lk(B,G)|=1$, the parallel strands link the twist box algebraically once. By Lemma \ref{lem:s-l}, $s(K_G)\le s(K_G^0)\le n-\sqrt{n}$.
\end{proof}

\begin{proposition}\label{prop:r-1}
Let $L[r]=\{(R,r),(B,0),(G,0)\}$ be an $n$-special RBG link, such that $R=U$, $r > 1$. Let $l$ be the linking number between $B$ and $G$. If $l^2 \le n$, then $s(K_G)\le n-\sqrt{n}$.
\end{proposition}
\begin{proof}
Consider the $(n-l^2)$-special RBG link $L[r-1]=\{(R,r-1),(B,0),(G,0)\}$. Since $r>1$, $U$ is $(r-2)$-slice in some $\#^m \overline{\mathbb{CP}^2}$. By Theorem \ref{thm:sinv}, $s(K_G[r-1]) \le n-l^2+1 - \sqrt{n-l^2+1}$. The knot $K_G$ of $L[r]$ can be obtained by adding one positive full twists along parallel strands from $K_G[r-1]$. Since $lk(B,G)=l$, the algebraic intersection number between the parallel strands and the twist box is $l$. By Lemma \ref{lem:s-l}, we have that $s(K_G) \le s(K_G[r-1])-|l|+l^2$. Thus, $s(K_G)\le n-|l|+1-\sqrt{n-l^2+1}$. If $l^2 \le n$, then $n-|l|+1-\sqrt{n-l^2+1}\le n-\sqrt{n}$. The result follows.
\end{proof}

\subsection{Experiments}
Consider an $n$-special RBG link with $R=U$. If $r \le 0$, then $R$ is $r$-slice in $\#^{|r|} {\mathbb{CP}^2}$. By Theorem \ref{thm:sinv}, the $n$-sliceness of $K_B$ implies that $s(K_G)\le n-\sqrt{n}$. Therefore, we cannot use Proposition \ref{Kai:s} to obstruct the $n$-sliceness of $K_G$. If $r>0$, then $R$ is $(r-1)$-slice in $\#^{r-1} \overline{\mathbb{CP}^2}$. By Theorem \ref{thm:sinv}, we only have $s(K_G)\le n+1-\sqrt{n+1}$. Now, pick an integer $n$ such that there exists an even integer $2q$ satisfying $n-\sqrt{n} < 2q \le n+1-\sqrt{n+1}$ (e.g. $n=3,6,8, 11, 13, 15$ etc.). If one can find an $n$-special RBG link with $s(K_G) = 2q$, then since $s(K_G) > n-\sqrt{n}$, we have that $K_G$ is not $n$-slice by Proposition \ref{Kai:s}. Hence, Theorem \ref{thm:sinv} leaves open the possibility to use the $s$-invariant to obstruct $K_G$ from being $n$-slice.

For example, we can consider the case where $n=3$. If there is a $3$-special RBG link with $R=U$ and $r>0$, such that $K_G$ is $3$-slice in some $\#^m\overline{\mathbb{CP}^2}$ and $s(K_B)=2$, then by Proposition \ref{Kai:s}, $K_B$ is not $3$-slice in any $\#^m\overline{\mathbb{CP}^2}$. Thus, such a link would produce an exotic $\#^m\overline{\mathbb{CP}^2}$. Note that $3=n=l(rl-2)$, so $(r,l)\in\{(5,1),(1, -1),(1,3)\}$. If $l=1$, then $r=5>1$, and by Proposition \ref{prop:r-1}, we have that $s(K_G)\le n-\sqrt{n}$, which cannot obstruct the $3$-sliceness of $K_G$. Thus, all potentially useful $3$-special RBG links are those with $r=1$ and $l \in \{-1,3\}$.

\begin{figure}[h]
{
   \fontsize{11pt}{13pt}\selectfont
   \def\svgwidth{2in}
   \begin{center}
   %% Creator: Inkscape 1.2.2 (b0a84865, 2022-12-01), www.inkscape.org
%% PDF/EPS/PS + LaTeX output extension by Johan Engelen, 2010
%% Accompanies image file 'luckypair.pdf' (pdf, eps, ps)
%%
%% To include the image in your LaTeX document, write
%%   \input{<filename>.pdf_tex}
%%  instead of
%%   \includegraphics{<filename>.pdf}
%% To scale the image, write
%%   \def\svgwidth{<desired width>}
%%   \input{<filename>.pdf_tex}
%%  instead of
%%   \includegraphics[width=<desired width>]{<filename>.pdf}
%%
%% Images with a different path to the parent latex file can
%% be accessed with the `import' package (which may need to be
%% installed) using
%%   \usepackage{import}
%% in the preamble, and then including the image with
%%   \import{<path to file>}{<filename>.pdf_tex}
%% Alternatively, one can specify
%%   \graphicspath{{<path to file>/}}
%% 
%% For more information, please see info/svg-inkscape on CTAN:
%%   http://tug.ctan.org/tex-archive/info/svg-inkscape
%%
\begingroup%
  \makeatletter%
  \providecommand\color[2][]{%
    \errmessage{(Inkscape) Color is used for the text in Inkscape, but the package 'color.sty' is not loaded}%
    \renewcommand\color[2][]{}%
  }%
  \providecommand\transparent[1]{%
    \errmessage{(Inkscape) Transparency is used (non-zero) for the text in Inkscape, but the package 'transparent.sty' is not loaded}%
    \renewcommand\transparent[1]{}%
  }%
  \providecommand\rotatebox[2]{#2}%
  \newcommand*\fsize{\dimexpr\f@size pt\relax}%
  \newcommand*\lineheight[1]{\fontsize{\fsize}{#1\fsize}\selectfont}%
  \ifx\svgwidth\undefined%
    \setlength{\unitlength}{288.03266039bp}%
    \ifx\svgscale\undefined%
      \relax%
    \else%
      \setlength{\unitlength}{\unitlength * \real{\svgscale}}%
    \fi%
  \else%
    \setlength{\unitlength}{\svgwidth}%
  \fi%
  \global\let\svgwidth\undefined%
  \global\let\svgscale\undefined%
  \makeatother%
  \begin{picture}(1,0.88916224)%
    \lineheight{1}%
    \setlength\tabcolsep{0pt}%
    \put(0,0){\includegraphics[width=\unitlength,page=1]{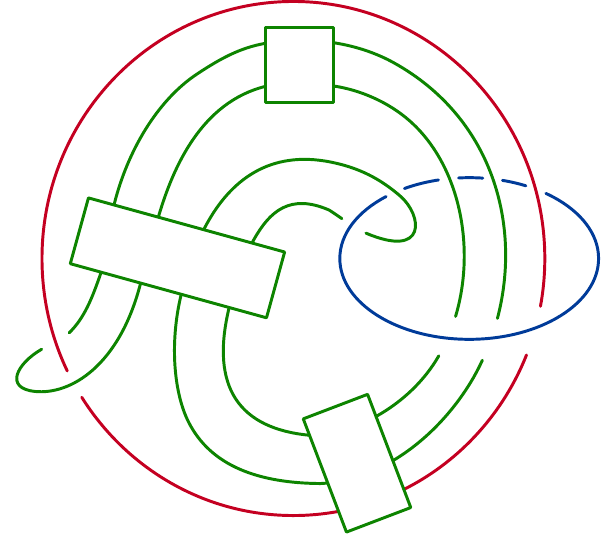}}%
    \put(0.48411544,0.76929057){\color[rgb]{0.05098039,0.51764706,0}\makebox(0,0)[lt]{\lineheight{1.25}\smash{\begin{tabular}[t]{l}$a$\end{tabular}}}}%
    \put(0.28760708,0.43817778){\color[rgb]{0.05098039,0.51764706,0}\makebox(0,0)[lt]{\lineheight{1.25}\smash{\begin{tabular}[t]{l}$b$\end{tabular}}}}%
    \put(0.57547219,0.10720367){\color[rgb]{0.05098039,0.51764706,0}\makebox(0,0)[lt]{\lineheight{1.25}\smash{\begin{tabular}[t]{l}$c$\end{tabular}}}}%
    \put(-0.00039656,0.60037211){\color[rgb]{0.76862745,0,0.12941176}\makebox(0,0)[lt]{\lineheight{1.25}\smash{\begin{tabular}[t]{l}$r$\end{tabular}}}}%
    \put(0.43478078,0.26002579){\color[rgb]{0.05098039,0.51764706,0}\makebox(0,0)[lt]{\lineheight{1.25}\smash{\begin{tabular}[t]{l}$0$\end{tabular}}}}%
    \put(0.93414507,0.28661302){\color[rgb]{0,0.23137255,0.6}\makebox(0,0)[lt]{\lineheight{1.25}\smash{\begin{tabular}[t]{l}$0$\end{tabular}}}}%
  \end{picture}%
\endgroup%

   \end{center}
   \caption{The diagram of $(r+2)$-special RBG links $L(a,b,c; r)$}
   \label{fig:sdiff3}
}
\end{figure}
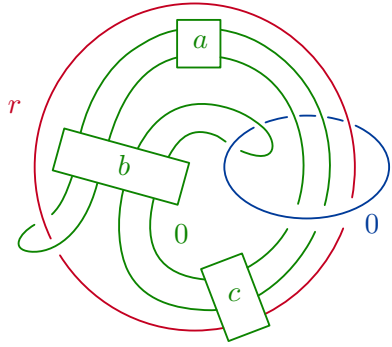

\begin{figure}
{
   \fontsize{11pt}{13pt}\selectfont
   \def\svgwidth{4in}
   \begin{center}
   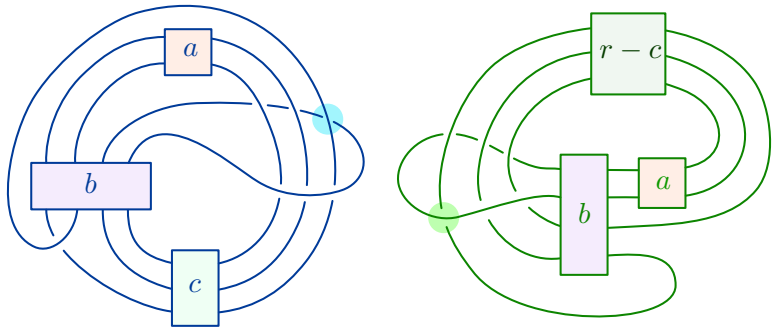
   \end{center}
   \caption{The diagrams of $K_B(a,b,c;r)$ and $K_G(a,b,c;r)$}
\label{KBGabc}
}
\end{figure}

Here is a potential situation that could lead to an exotic $\#^m\overline{\mathbb{CP}^2}$ using the $s$-invariant. Consider the family $L(a,b,c;r)$ of special RBG links in Figure \ref{fig:sdiff3}. Since $l=-1$, the link $L(a,b,c;r)$ is an $(r+2)$-special RBG link. Figure \ref{KBGabc} gives the diagrams of $K_B(a,b,c;r)$ and $K_G(a,b,c;r)$. Note that two diagrams in Figure \ref{KBGabc} are isotopic for all choices of $a,b$ and $c$. Thus, we denote $K_B(a,b,c;r)$ by $K(a,b,c)$, and $K_G(a,b,c;r)$ by $K(a,b,r-c)$.

\begin{example} \label{eg:luckylink}
The link $L(-2,1,-2;1)$ (cf. Figure \ref{fig:sdiff3}) is a $3$-special RBG link. Using KnotJob \cite{KnotJob}, one can compute that $s(K_G) = 0$ and $s(K_B) = 2$. The knot $K_B(-2, 1, -2;1)=K(-2,1,-2)$ is recognized by SnapPy \cite{SnapPy} as $K9\_533$, and $K_G(-2, 1, -2;1)=K(-2, 1, 3)$ has a diagram with $28$ crossings as in Figure \ref{intro} (left) by KnotJob \cite{KnotJob}.
\end{example}

From Example \ref{eg:luckylink}, we obtain the following theorem.
{
\renewcommand{\thetheorem}{\ref{potential3}}
\begin{theorem}
If the knot $K(-2,1,3)$ from the left-hand side of Figure \ref{intro} is $3$-slice in some $\#^m \overline{\mathbb{CP}^2}$, then there exists an exotic $\#^m \overline{\mathbb{CP}^2}$.
\end{theorem}
\addtocounter{theorem}{-1}
}

Therefore, we are interested in whether $K(-2, 1, 3)$ is $3$-slice in some $\#^m\overline{\mathbb{CP}^2}$. Notice that if we change the colored crossings in Figure \ref{KBGabc} by adding $2$-handles, we get ribbon knots for any choice of $(a,b,c;r)$. Therefore, both $K_B$ and $K_G$ are $H$-slice in $\mathbb{CP}^2$ and $4$-slice in $\overline{\mathbb{CP}^2}$. Using SnapPy \cite{SnapPy}, we compute that $\tau(K(-2,1,3))=0$, which does not obstruct $3$-sliceness. However, by considering $L(-2, 1, -3;0)$, we can obstruct $2$-sliceness.
\begin{proposition}\label{2-slice}
The knot $K(-2,1,3)$ in Figure \ref{intro} (left) is not $2$-slice in any $\#^m\overline{\mathbb{CP}^2}$.
\end{proposition}
\begin{proof}
The knot $K(-2, 1, 3)$ has the same diagram as $K_G(-2,1,-3;0)$. Consider the $2$-special RBG link in the $4$-parameter family in Figure \ref{fig:sdiff3} with $(a,b,c;r)=(-2, 1, -3;0)$. Using KnotJob, we compute that $s(K_B(-2,1,-3;0)) =2$. However, by Theorem \ref{thm:sinv}(a), if $K_G(-2,1,-3;0)$ is $2$-slice, then $s(K_B(-2,1,-3;0))\le 2-\sqrt{2}<2$. Thus,  $K(-2, 1, 3)= K_G(-2,1,-3;0)$ is not $2$-slice in any $\#^m\overline{\mathbb{CP}^2}$.
\end{proof}

Although we cannot determine whether $K(-2,1,3)$ is $3$-slice, Proposition \ref{2-slice} gives an example where we can use an $n$-special RBG link to obstruct the $n$-sliceness of a knot. Another example of this type is the $2$-special RBG link in Figure \ref{fig:sdifferent}.

\begin{figure}[h]
{
   \fontsize{9pt}{11pt}\selectfont
   \def\svgwidth{2.5in}
   \begin{center}
   %% Creator: Inkscape 1.2 (dc2aeda, 2022-05-15), www.inkscape.org
%% PDF/EPS/PS + LaTeX output extension by Johan Engelen, 2010
%% Accompanies image file 'sdifferent.pdf' (pdf, eps, ps)
%%
%% To include the image in your LaTeX document, write
%%   \input{<filename>.pdf_tex}
%%  instead of
%%   \includegraphics{<filename>.pdf}
%% To scale the image, write
%%   \def\svgwidth{<desired width>}
%%   \input{<filename>.pdf_tex}
%%  instead of
%%   \includegraphics[width=<desired width>]{<filename>.pdf}
%%
%% Images with a different path to the parent latex file can
%% be accessed with the `import' package (which may need to be
%% installed) using
%%   \usepackage{import}
%% in the preamble, and then including the image with
%%   \import{<path to file>}{<filename>.pdf_tex}
%% Alternatively, one can specify
%%   \graphicspath{{<path to file>/}}
%% 
%% For more information, please see info/svg-inkscape on CTAN:
%%   http://tug.ctan.org/tex-archive/info/svg-inkscape
%%
\begingroup%
  \makeatletter%
  \providecommand\color[2][]{%
    \errmessage{(Inkscape) Color is used for the text in Inkscape, but the package 'color.sty' is not loaded}%
    \renewcommand\color[2][]{}%
  }%
  \providecommand\transparent[1]{%
    \errmessage{(Inkscape) Transparency is used (non-zero) for the text in Inkscape, but the package 'transparent.sty' is not loaded}%
    \renewcommand\transparent[1]{}%
  }%
  \providecommand\rotatebox[2]{#2}%
  \newcommand*\fsize{\dimexpr\f@size pt\relax}%
  \newcommand*\lineheight[1]{\fontsize{\fsize}{#1\fsize}\selectfont}%
  \ifx\svgwidth\undefined%
    \setlength{\unitlength}{221.22719561bp}%
    \ifx\svgscale\undefined%
      \relax%
    \else%
      \setlength{\unitlength}{\unitlength * \real{\svgscale}}%
    \fi%
  \else%
    \setlength{\unitlength}{\svgwidth}%
  \fi%
  \global\let\svgwidth\undefined%
  \global\let\svgscale\undefined%
  \makeatother%
  \begin{picture}(1,0.63737829)%
    \lineheight{1}%
    \setlength\tabcolsep{0pt}%
    \put(0,0){\includegraphics[width=\unitlength,page=1]{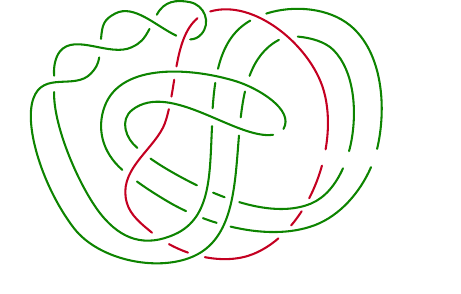}}%
    \put(0.49796958,0.01005401){\color[rgb]{0.76862745,0,0.12941176}\makebox(0,0)[lt]{\lineheight{1.25}\smash{\begin{tabular}[t]{l}$0$\end{tabular}}}}%
    \put(0.92234765,0.38050764){\color[rgb]{0,0.23921569,0.6}\makebox(0,0)[lt]{\lineheight{1.25}\smash{\begin{tabular}[t]{l}$0$\end{tabular}}}}%
    \put(-0.00019598,0.30360862){\color[rgb]{0.05098039,0.51764706,0}\makebox(0,0)[lt]{\lineheight{1.25}\smash{\begin{tabular}[t]{l}$0$\end{tabular}}}}%
    \put(0,0){\includegraphics[width=\unitlength,page=2]{sdifferent.pdf}}%
  \end{picture}%
\endgroup%

   \end{center}
   \caption{A diagram of the $2$-special RBG link $L(-2,1,2;0)$ from the family in Example \ref{eg:luckylink}}
   \label{fig:sdifferent}
}
\end{figure}

\begin{example}\label{sdifferent}
The $2$-special RBG link $L(-2,1,2;0)$ in Figure \ref{fig:sdifferent} gives another example where $s(K_B)\neq s(K_G)$. Using SnapPy \cite{SnapPy}, we recognize $K_G$ as $K9$\verb|_|$533$. Figure \ref{intro} (right) gives a diagram of $K_B$. Using KnotJob \cite{KnotJob}, we compute that $s(K_B)=0$, $s(K_G)=2$. Also, using SnapPy \cite{SnapPy}, we have that $\tau(K_B)=\tau(K_G)=0$.
\end{example}
Using the $2$-RBG link in Figure \ref{fig:sdifferent}, we obtain the following:

\begin{proposition}\label{2-slice2}
The knot $K(-2,1,2)$ in Figure \ref{intro} (right) is not $2$-slice in any $\#^m\overline{\mathbb{CP}^2}$.
\end{proposition}

\begin{proof}
Since $K(-2,1,2)$ is the blue knot $K_B(-2, 1, 2; 0)$ associated to the $2$-special RBG link $L(-2,1,2;0)$ in Example \ref{sdifferent} (see Figure \ref{fig:sdifferent}), we apply Theorem \ref{thm:sinv} (a) with $n=2$. Since $s(K_G)=2$ is larger than $n-\sqrt{n}=2-\sqrt{2}$, the associated knot $K_B$ is not $2$-slice in any $\#^m\overline{\mathbb{CP}^2}$. Thus, $K(-2,1,2)$ is not $2$-slice in any $\#^m\overline{\mathbb{CP}^2}$.
\end{proof}
So far, we have proved Theorem \ref{thm:2slice}, which is a combination of Proposition \ref{2-slice} and Proposition \ref{2-slice2} using Theorem \ref{thm:sinv}. We do not know other proofs of Theorem \ref{thm:2slice}. In general, there are two other approaches to obstruct knots from being slice $n$-slice in any $\#^m\overline{\mathbb{CP}^2}$. The first one is to use the $\tau$ invariant from \cite{OS03}, where Ozsv{\'{a}}th and Szab{\'{o}} proved that the $\tau$ invariant of a knot $K$ satisfies the adjunction inequality in a negative definite $4$-manifold $W$ with $b_1(W)=0$, namely
$$
2\tau(K)\le 2g(\Sigma) - |[\Sigma]|-[\Sigma]\cdot[\Sigma],
$$
where $\Sigma$ is a properly, smoothly embedded surface in $W$ without closed components, such that $\partial\Sigma=K$. Similarly to the proof of Theorem \ref{Kai:s}, we have that if $K$ is $n$-slice in some $\#^m\overline{\mathbb{CP}^2}$, then $2\tau(K)\le n-\sqrt{n}$.
However, we have $s=\tau=0$ for the knots in Figure \ref{intro}, so this method does not apply to Theorem \ref{thm:2slice}. 

The second approach is to use the $d$-invariant from \cite{OSdinv}. Theorem 9.6 in  \cite{OSdinv} gives a lower bound on the $d$-invariants of a rational homology sphere that bounds a definite 4-manifold. Moreover, by Proposition 1.6 in \cite{lensdinv}, we can compute the $d$-invariant of $S_2^3(K)$ by looking at the full knot Floer complex of $K$. A calculation of $\widehat{\text{HFK}}$ with SnapPy \cite{SnapPy} does not quite determine the full knot complexes of the knots in Figure \ref{intro}, but suggests that they could be CFK-equivalent to the unknot. Therefore, this method does not apply either. 

\section{$n$-peculiar RBG links}

\label{npeculiar}
In this section, we consider a different construction of $|n|$-RBG links (that are usually not $n$-special). Let $n\in\mathbb{Z}$ and $r\in\mathbb{Q}$. Consider a two component link $\{(K,n),(\mu_K, r)\}$, where $\mu_K$ is a meridian of $K$. By a Rolfsen twist, there exists a homeomorphism from $S^3_{n,r}(K,\mu_K)$ to $S^3_{n-1/r}(K)$, which restricts to the identity outside a tubular neighborhood of $K$. If  $K=U$ and $n-1/r= 1/t$ for some $t\in \mathbb{Z}$, then $\{(K,n),(\mu_K, r)\}$ is a surgery diagram for $S^3$.

{
\renewcommand{\thetheorem}{\ref{def:pecu}}
\begin{definition}
A link $\{(R,r),(B,b),(G,g)\}$ is called an $n$-peculiar RBG link, if there exists $t\in\mathbb{Z}$ such that
\begin{itemize}
\item $R=U$ and $B,G$ are meridians of $R$,
\item $b=g= 1/r + 1/t$, 
\item $n=(g+b-2l) - t(l-b)^2$, 
\end{itemize}
where $l=lk(B,G)$ under an orientation of $L$ such that $lk(B,R)=lk(G,R)=1$.
\end{definition}
\addtocounter{theorem}{-1}
}
We can read out the diagram of $K_G$ from its RBG link diagram in three steps. Let $\Delta_R$ be a meridian disk of $B$, such that  $\partial\Delta_R = R$. First, isotope $G$ away from $\Delta_R$ by sliding $G$ over $B$. Then, slam dunk $R$ into $B$, and obtain a link diagram where $R$ is deleted and $B$ has framing $1/t$. Finally, blow down the blue component $B$ by Rolfsen twists. We can get a diagram for $K_B$ in a similar way.
\begin{lemma}
An $n$-peculiar RBG link induces an $n$-surgery homeomorphism $\phi: S_n^3(K_B)\rightarrow S_n^3(K_G)$.
\end{lemma}
\begin{proof}
We keep track of the linking matrix $M_L$ of the link $L$ under those three diagram changes. We start with
$$M_L=\begin{bmatrix}
r &1 &1 \\
1 &b &l \\
1 &l  &g
\end{bmatrix}.$$
Sliding $G$ over $B$ so that $|\Delta_R\cap G|=0$, the linking matrix becomes $$\begin{bmatrix}
r &1 &0 \\
1 &b &l-b \\
0 &l-b  &g+b-2l
\end{bmatrix}.$$ 
After the slam-dunk of $R$ into $B$, the linking matrix is $$\begin{bmatrix}
b- 1/r &l-b \\
l-b  &g+b-2l
\end{bmatrix}.$$
Now, blow down the blue component by a Rolfsen twist along $B$. The framing of $K_G$ is $$f_g = (g+b-2l) - t(l-b)^2=n.$$
\end{proof}

\begin{figure}[h]
{
   \fontsize{9pt}{11pt}\selectfont
   \def\svgwidth{2.2in}
   \begin{center}
   %% Creator: Inkscape 1.2.2 (b0a84865, 2022-12-01), www.inkscape.org
%% PDF/EPS/PS + LaTeX output extension by Johan Engelen, 2010
%% Accompanies image file '0peculiar.pdf' (pdf, eps, ps)
%%
%% To include the image in your LaTeX document, write
%%   \input{<filename>.pdf_tex}
%%  instead of
%%   \includegraphics{<filename>.pdf}
%% To scale the image, write
%%   \def\svgwidth{<desired width>}
%%   \input{<filename>.pdf_tex}
%%  instead of
%%   \includegraphics[width=<desired width>]{<filename>.pdf}
%%
%% Images with a different path to the parent latex file can
%% be accessed with the `import' package (which may need to be
%% installed) using
%%   \usepackage{import}
%% in the preamble, and then including the image with
%%   \import{<path to file>}{<filename>.pdf_tex}
%% Alternatively, one can specify
%%   \graphicspath{{<path to file>/}}
%% 
%% For more information, please see info/svg-inkscape on CTAN:
%%   http://tug.ctan.org/tex-archive/info/svg-inkscape
%%
\begingroup%
  \makeatletter%
  \providecommand\color[2][]{%
    \errmessage{(Inkscape) Color is used for the text in Inkscape, but the package 'color.sty' is not loaded}%
    \renewcommand\color[2][]{}%
  }%
  \providecommand\transparent[1]{%
    \errmessage{(Inkscape) Transparency is used (non-zero) for the text in Inkscape, but the package 'transparent.sty' is not loaded}%
    \renewcommand\transparent[1]{}%
  }%
  \providecommand\rotatebox[2]{#2}%
  \newcommand*\fsize{\dimexpr\f@size pt\relax}%
  \newcommand*\lineheight[1]{\fontsize{\fsize}{#1\fsize}\selectfont}%
  \ifx\svgwidth\undefined%
    \setlength{\unitlength}{283.89884663bp}%
    \ifx\svgscale\undefined%
      \relax%
    \else%
      \setlength{\unitlength}{\unitlength * \real{\svgscale}}%
    \fi%
  \else%
    \setlength{\unitlength}{\svgwidth}%
  \fi%
  \global\let\svgwidth\undefined%
  \global\let\svgscale\undefined%
  \makeatother%
  \begin{picture}(1,0.69959973)%
    \lineheight{1}%
    \setlength\tabcolsep{0pt}%
    \put(0,0){\includegraphics[width=\unitlength,page=1]{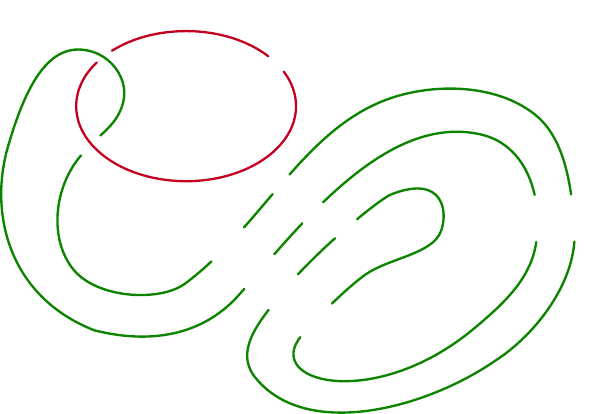}}%
    \put(0.27165122,0.67520448){\color[rgb]{0.76862745,0,0.1254902}\makebox(0,0)[lt]{\lineheight{1.25}\smash{\begin{tabular}[t]{l}$1/2$\end{tabular}}}}%
    \put(0.05010016,0.10303363){\color[rgb]{0.05098039,0.51764706,0}\makebox(0,0)[lt]{\lineheight{1.25}\smash{\begin{tabular}[t]{l}$1$\end{tabular}}}}%
    \put(0.80462447,0.66458648){\color[rgb]{0,0.23137255,0.6}\makebox(0,0)[lt]{\lineheight{1.25}\smash{\begin{tabular}[t]{l}$1$\end{tabular}}}}%
    \put(0,0){\includegraphics[width=\unitlength,page=2]{0peculiar.pdf}}%
    \put(0.44187357,0.20412791){\color[rgb]{0.05098039,0.51764706,0}\makebox(0,0)[lt]{\lineheight{1.25}\smash{\begin{tabular}[t]{l}$b$\end{tabular}}}}%
    \put(0.91879275,0.31554923){\color[rgb]{0.05098039,0.51764706,0}\makebox(0,0)[lt]{\lineheight{1.25}\smash{\begin{tabular}[t]{l}$a$\end{tabular}}}}%
  \end{picture}%
\endgroup%

   \end{center}
   \caption{A family of $0$-peculiar links}
   \label{0pecu_s}
}
\end{figure}
\begin{example}
Consider the family of peculiar RBG links in Figure \ref{0pecu_s}, parametrized by two twisting boxes. Since $l=1$ and $b=g=1$, we have $n=0$. When $(a,b)$ is $(2,-1)$ or $(3, -2)$, the peculiar link $L(a,b)$ generates a knot pair such that $s(K_G) = -2$ and $s(K_B) = 0$. (When $(a,b)=(2,-1)$, the knot $K_B$ is $11_{270}$.) However, since the signature of $K_B$ is $2$ in each case, the knots $K_B$ are not $H$-slice in any $\#^m\mathbb{CP}^2$, so the knot pairs do not produce any exotic $\#^m\mathbb{CP}^2$.
\end{example}

\bibliography{n-surgery}
\bibliographystyle{custom}
\end{document}